\documentclass[final]{siamltex}

\usepackage{amssymb}
\usepackage{amsmath}
\usepackage{enumerate}
\usepackage{amstext}
\usepackage{graphicx}

\makeatletter\@addtoreset{equation}{section}\makeatother
\def\theequation{\arabic{section}.\arabic{equation}}

\allowdisplaybreaks

\DeclareSymbolFont{msbm}{U}{msb}{m}{n}
\DeclareMathSymbol{\N}{\mathalpha}{msbm}{'116}
\DeclareMathSymbol{\R}{\mathalpha}{msbm}{'122}

\renewcommand{\.}{\cdot}
\newcommand{\Lip}{\text{Lip}}
\renewcommand{\P}{\mathcal{P}}
\newcommand{\K}{\mathcal{K}}
\newcommand{\C}{\mathcal{C}}
\newcommand{\T}{\mathcal{T}}
\newcommand{\E}{\mathcal{E}}
\newcommand{\B}{\mathcal{B}}
\newcommand{\D}{\mathcal{D}}
\newcommand{\bm}{}

\newtheorem{remark}{Remark}

\newcommand{\np}{_{n+1}}
\newcommand{\hnp}{^{n+1}}
\newcommand{\nph}{_{n+\frac{1}{2}}}

\newcommand{\dt}{\Delta t}
\newcommand{\dx}{\Delta x}
\newcommand{\M}{\mathcal{M}}
\renewcommand{\S}{\mathbb{S}}
\newcommand{\dist}{W_1}

\begin{document}

\bibliographystyle{plain}

\title{A retarded mean-field approach for interacting fiber structures}

\author{R. Borsche\footnotemark[1] \and A. Klar\footnotemark[1]
  \footnotemark[2] \and  C. Nessler\footnotemark[1] \and A. Roth\footnotemark[1]  \and O. Tse\footnotemark[1]}
\footnotetext[1]{Fachbereich Mathematik, Technische Universit\"at
  Kaiserslautern, Germany \\ 
  \{borsche@mathematik.uni-kl.de, klar@mathematik.uni-kl.de, nessler@mathematik.uni-kl.de, roth@mathematik.uni-kl.de, tse@mathematik.uni-kl.de\}}
\footnotetext[2]{Fraunhofer ITWM, Kaiserslautern, Germany \\
  \{klar@itwm.fhg.de, wegener@itwm.fhg.de\}}

\maketitle



\begin{abstract}
We consider an interacting system of one-dimensional structures modelling fibers with fiber-fiber interaction in a fiber lay-down process. The resulting microscopic system is investigated by looking at different asymptotic limits of the corresponding stochastic model. Equations arising from mean-field and diffusion limits are considered. Furthermore, numerical methods for the stochastic system and its mean-field counterpart are discussed. A numerical comparison of solutions corresponding to the different scales (microscopic, mesoscopic and macroscopic) is included.
\end{abstract}

{\bf Keywords} interacting stochastic particles, fibers, mean-field equations, retarded potential, delay equations.

{\bf 2010 AMS Subject Classification:} 92D50, 35B40, 82C22, 92C15

\section{Introduction}\label{section1}

One-dimensional structures appear in various context of industrial applications. They are used, for example, in the modelling of polymers in suspensions, composite materials, nanostructures, fiber dynamics in turbulent flows and, in particular, fiber lay-down in technical processes of non-woven materials. Furthermore, such structures have been modelled on different levels of description involving different scales. Besides microscopic models, mesoscopic kinetic or Fokker--Planck equations have been widely used for a statistical description of the fiber or polymer distributions. We refer to \cite{BS07,BS11} and \cite{GKMW07,KMW12} for concrete examples in the industry. In this article, we consider non-woven materials, which are webs of long flexible fibers. Production processes and models corresponding to the lay-down of such fibers have been intensively investigated. See \cite{BGKMW07} and the above cited references.

In the above mentioned investigations, the one-dimensional structures (fibers) under consideration were assumed to be mutually independent, which clearly does not represent reality. Therefore, the present work aims at including fiber-fiber interaction, thereby describing the size of each fiber and, simultaneously, the absence of intersection among fibers. This is achieved by simply including the interaction of structures into a well investigated model described in non-woven production processes. Taking into account the interaction of the structures on the microscopic level leads to coupled systems of stochastic differential equations. Its statistical description should also take into account the interactions and will consequently no longer be based on the classical Fokker--Planck model.

The new model makes use of a microscopic systems of retarded stochastic differential equations, and its mesoscopic description is obtained via formal mean-field procedures. The mean-field limit is described by a McKean--Vlasov type equation with a delay term. We perform an analytical investigation of the mean field limit, as well as a numerical comparison of microscopic, mean field and macroscopic equations. The analysis of the limit is based on the work in \cite{BH,dobru,G03,G12,GMR14,neunzert,spohn2}. For numerical methods for mean-field type equations we refer to \cite{BS11,Sem3,RKSZ14,Sem6}.

 The paper is organized as follows: starting from a model for independent fibers, we present in Section~\ref{interacting} a new model for interacting fibers, which takes into account the finite size of the fibers and prevents intersection of the fibers. The model is based on a system of retarded stochastic differential equations with suitable interaction potentials. Section~\ref{meanfieldsection} describes the mean-field equation and a discussion of its stationary solutions. The core of this section is devoted to the proof of the mean-field limit for the corresponding deterministic system, i.e., the rigorous derivation of the mean-field equation from the system of retarded deterministic equations. Section~\ref{diffusionsection} contains the diffusion limit of the kinetic equation, while Section~\ref{numerics} contains a description of the numerical methods used for the microscopic and mean-field equations. The numerical results include an investigation and comparison of stationary states, as well as a convergence analysis of solutions to equilibrium for both the microscopic and mean-field equations. We finally conclude in Section~\ref{conclusion}.


\section{Interacting isotropic fiber models}
\label{interacting}

We begin by reviewing a basic, mutually independent fiber model for the lay-down process of fibers, described by a stochastic dynamical system in dimensions $d \geq 2$ (cf.~\cite{BGKMW07,KMW12,KST14}). Using the state space $\M:=\R^d\times \S^{d-1}$, where $\S^{d-1}$ denotes the unit sphere in $\R^d$, a fiber is given by a path of the following stochastic differential equation:
\begin{align}\label{2dmodelcoordfree}
 \begin{aligned}
  d x_t &=\tau_t\,dt,\\
  d\tau_t &= \big(I - \tau_t \otimes  \tau_t\big) \circ \Big(- \frac{1}{d-1} \nabla_x V(x_t)\,dt  + A\,d W_t\Big)
 \end{aligned}
\end{align}
with initial condition $(x_0,\tau_0) \in \M$. Here, $V\colon\R^d \rightarrow \R$ is the so-called coiling potential, $A$ a nonnegative diffusion coefficient and $W_t$ the standard Brownian motion. We use a coordinate free formulation, where $\circ$ denotes the Stratonovich integral. Note also that $(I-\tau\otimes\tau)$ is the projector of $\R^d$ onto the unit sphere $\S^{d-1}$. We refer to \cite{DM1} for related work. We equip our state space with the measure $dx\,d\nu$, where $d x$ denotes the Lebesgue measure on $\R^d$ and $d \nu$ the normalized surface measure on $\S^{d-1}$.
Note that the dimensional scaling $1/(d-1)$ is introduced for convenience, in order to achieve a stationary state which does not explicitly depend on the spatial dimension.

The fiber model above only describes the evolution of the center line of the fiber and does not capture the effects of the finite size of the fibers. Hence, self-intersection and intersection among fibers are not prevented in the model. A possibility to remedy this deficiency it to include hysteresis into the system, which enables the system to prevent self-intersection, and for multiple fibers, intersection among them. Considerations along this line of thought lead to the following interacting fiber model. 

Consider $N \in \N$ fibers with position and velocity $(x^i,\tau^i) \in \M$, for each $i \in \{1,\dots, N \}$. The interacting fiber model reads
%
\begin{align}\label{2dmodelcoordfreeinter}
\begin{aligned}
 d x_t^i&=\tau_t^i\,dt \\
 d\tau_t^i&= \big(I - \tau_t^i \otimes  \tau_t^i\big) \circ \bigg( - \frac{1}{d-1} \nabla_{x} V( x_t^i )\,dt \\
    &\hspace{7em} - \frac{1}{d-1} \bigg[\frac{1}{N} \sum_{ j=1}^N \frac{1}{t} \int_0^t \nabla_{x}U( x_t^i- x_s^j)\,ds \bigg]dt
   + A \, d W^i_{t} \bigg),
\end{aligned} 
\end{align}
with independent Brownian motions $W_t^i$. In comparison to the previous fiber model, we include a scaled, nonlocal (in time) interaction term 
\[
  \frac{1}{t} \int_0^t \nabla_{x}U( x_t^i- x_s^j)\,ds,
\] 
where $U$ is an interaction potential, which is repulsive in our case, so as to avoid any contact among the fibers. Note that a 'non-retarded' version of this system would be  similar to a model for swarming with roosting potential (cf.~\cite{CKMT,CDP}). In addition to the interaction term, our new model includes a delay term, i.e., an integration with respect to time, to describe the fact that fibers interact with each other and with itself on the whole fiber length. The factor $1/N$ leads to the so called   'weak coupling scaling' (cf.~\cite{BH,neunzert}).  Similiarly, the scaling $1/t$ is a normalization of the potential with respect to the time integration. Also note that the summation does not exclude $i = j$, which accounts for self-interaction of a fiber with itself.
  
There is some freedom in the choice of the interaction potential. As a simple example, one can use a mollifier type potential
\[
 U(x)=U(| x |) = C \exp\Big(-\frac{(2R)^2}{(2R)^2- | x |^2}\Big),\quad\text{for}\ | x | < 2R ,
\]
where $R$ is a nonnegative parameter representing the fiber radius and $C>0$ is a fixed constant describing the strength of the interaction. Alternatively, a potential could be described by a smoothed version of Heaviside type potential
\[ 
 U(x)=U(| x |) = C \Theta(2R-| x | ),
\]
with Heaviside function $\Theta$. A smooth version of such a potential may be given by 
\begin{eqnarray}\label{Uinteraction3}
U(x) = \frac{C}{1+ \exp\Big(-k\left( 1 - \frac{| x |^2}{(2R)^2}\right)\Big)},
\end{eqnarray}
for some regularizing parameter $k>0$.

\begin{figure}[h]
\centering
\hspace*{1em}\includegraphics[trim=15.0cm 1.0cm 11cm 0cm, clip, width=0.4\textwidth]{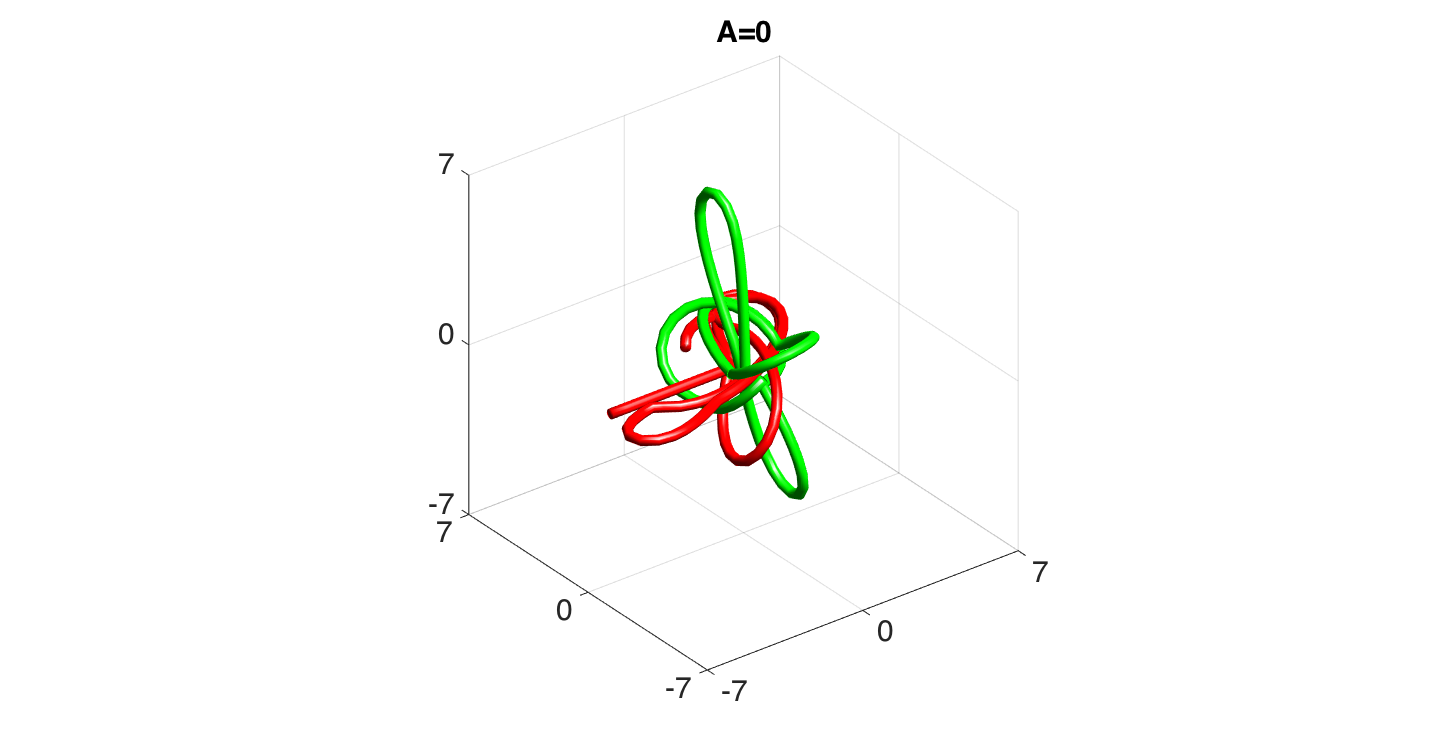}
\hspace*{1.5em}\includegraphics[trim=15.0cm 1.0cm 11cm 0cm, clip, width=0.4\textwidth]{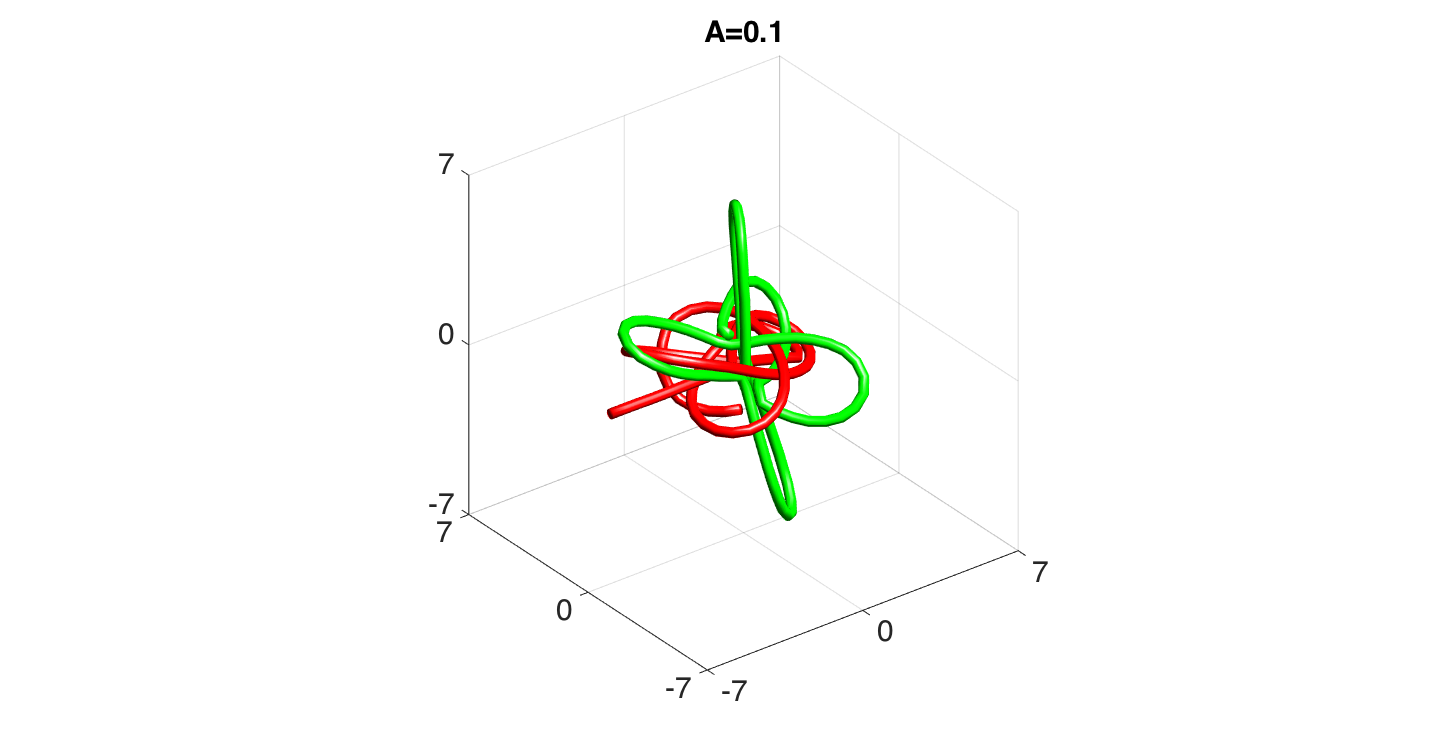}\\[1em]
\hspace*{1em}\includegraphics[trim=15.0cm 1.0cm 11cm 0cm, clip, width=0.4\textwidth]{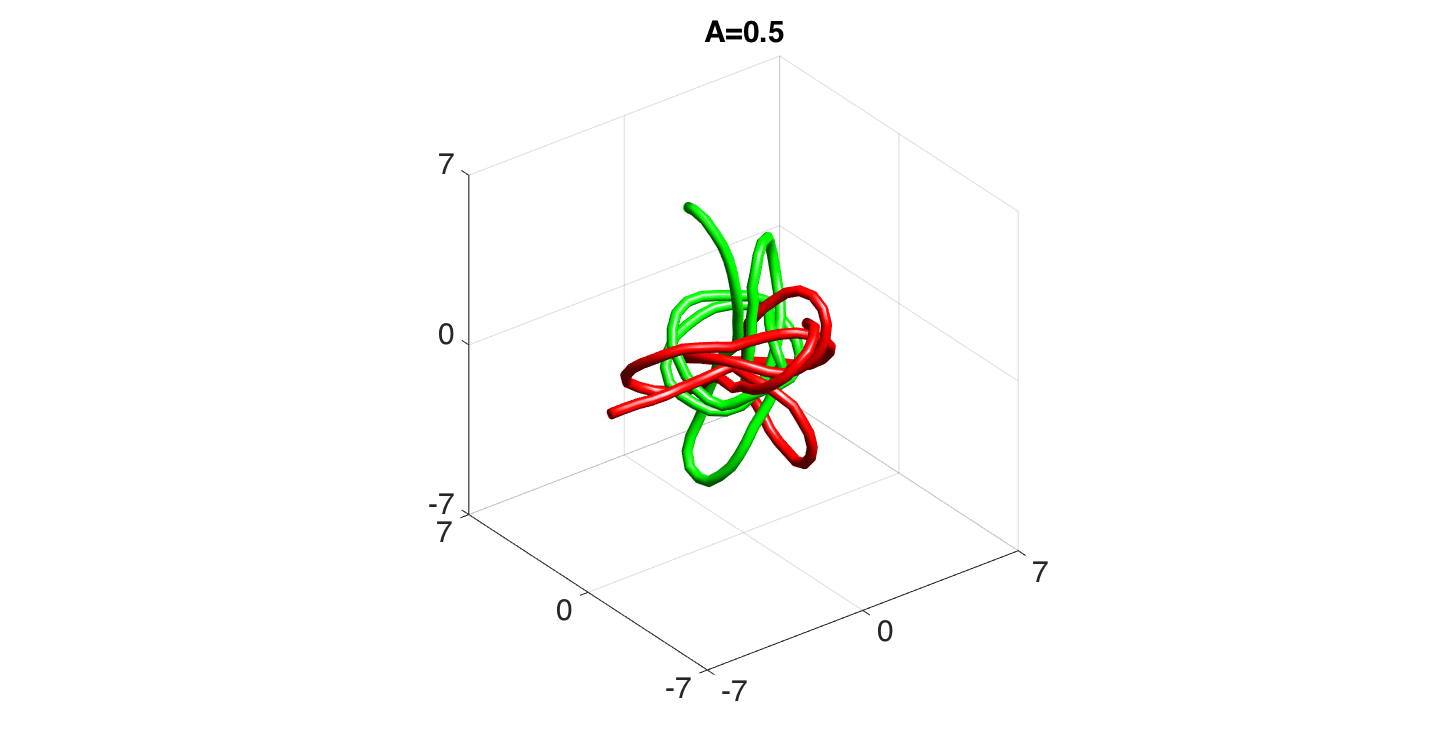}
\hspace*{1.5em}\includegraphics[trim=15.0cm 1.0cm 11cm 0cm, clip, width=0.4\textwidth]{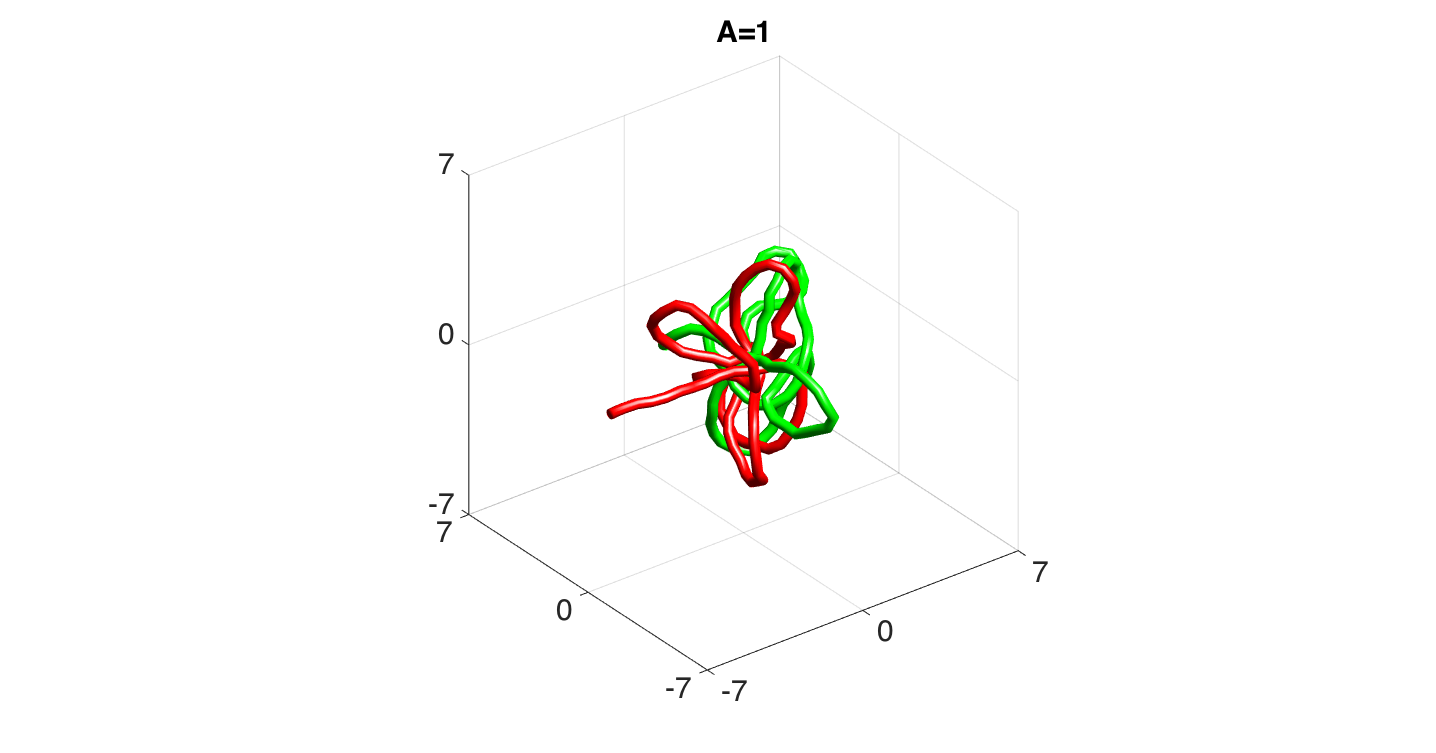}
\caption{Comparison of the influence of the noise $A$ for two interacting fibers.}
\label{compNoise}
\end{figure}  

In Figure~\ref{compNoise} we compare the fiber curves for different noise amplitudes $A$. We use the coiling potential $V(x) = \frac{1}{2}| x|^2$ and the interaction potential $U$ from (\ref{Uinteraction3}), with $k=100,\ C = 10$ and $R=0.4$. Since the computation was done for a short amount of time, we neglect the scaling $1/t$ in front of the interaction part. We observe that the non-intersecting fiber curves are increasingly altered with increasing noise amplitude.

\begin{remark}
\begin{enumerate}
 \item For $U \equiv 0$, one obtains a fully decoupled system for $(x^i, \tau^i)$, and each fiber is described by the mutually independent model given in (\ref{2dmodelcoordfree}).
 
 \item To consider inelastic  interactions, one could include damping terms depending on the velocity in the interaction potential. This would lead to equations 
 where a dissipative force is included in the interaction term.
 
 \item As in the case without interaction, it is also possible to formulate a smooth version of the interacting fiber system. The basic idea is to replace the Brownian motion on the sphere by an Ornstein--Uhlenbeck process (cf.~\cite{KMW12b}). 
\end{enumerate}
\end{remark}

\begin{remark}
It is also possible to include reference curves $\gamma$ into the model, which describes, for example, the motion of a conveyor belt. This can be done in the following way. We denote with $\eta^i: \R_+ \rightarrow \R^d$ the actual fiber curves. Then (\ref{2dmodelcoordfreeinter}) changes to
\begin{align}\label{modelreference}
\begin{aligned}
 d \eta_t^i&=\tau_t^i\,dt \\
 d\tau_t^i&= \big(I - \tau_t^i \otimes  \tau_t^i\big) \circ \bigg( - \frac{1}{d-1} \nabla_{x} V( \eta_t^i -\gamma)\,dt \\
    &\hspace{7em} - \frac{1}{d-1} \bigg[\frac{1}{N} \sum_{ j=1}^N \frac{1}{t} \int_0^t \nabla_{x}U( \eta_t^i- \eta_s^j)\,ds \bigg]dt
   + A \, d W^i_{t} \bigg).
\end{aligned}
\end{align}
 A change of variables $\xi_i:= \eta_t^i - \gamma$ describes the deviation of the fiber curves from the reference curve and thus, (\ref{modelreference}) may be written as
\begin{align}\label{modelreference2}
\begin{aligned}
 d \xi_t^i&=\tau_t^i\,dt - d\gamma_t\\
 d\tau_t^i&= \big(I - \tau_t^i \otimes  \tau_t^i\big) \circ \bigg( - \frac{1}{d-1} \nabla_{x} V( \xi_t^i)\,dt \\
    &\hspace{2em} - \frac{1}{d-1} \bigg[\frac{1}{N} \sum_{ j=1}^N \frac{1}{t} \int_0^t \nabla_{x}U( \xi_t^i- \xi_s^j + \gamma_t - \gamma_s)\,ds \bigg]dt
   + A \, d W^i_{t} \bigg).
\end{aligned}
\end{align}
In this case the force due to the interaction potential depends on the relative fiber point position as well as on the difference $\gamma_t - \gamma_s$ in the reference curve.

In the relevant case of non-wovens on a conveyor belt moving with constant speed, we consider $d\in\{2,3\}$ and a reference curve given by $\gamma_t = - v_{ref} e_1 t$. Here $v_{ref} = v_{belt}/ v_{prod}$ is the ratio between the speed of the conveyor belt, $v_{belt}$, and the speed of the production process, $v_{prod}$, and $e_1$ denotes the direction of belt movement (cf.~\cite{KMW12}).
\end{remark}

In the following, we formulate a slightly more general model than (\ref{2dmodelcoordfreeinter}). Since, in reality, the fiber material is transported away by the belt, interaction will not take place for the full history of the fibers. Thus, it is reasonable to consider a cut-off with a cut-off size $H>0$. Define 
\[
 h(t) = \begin{cases} t & \text{for }t\le H \\ H & \text{for }t>H\end{cases},\quad H\in(0,\infty).
\] 
Then the interacting fiber model with cut-off is given by
\begin{align}\label{modelshortdelay}
\begin{aligned}
 d x_t^i&=\tau_t^i\,dt \\
 d\tau_t^i&= \big(I - \tau_t^i \otimes  \tau_t^i\big) \circ \bigg( - \frac{1}{d-1} \nabla_{x} V( x_t^i )\,dt \\
    &\hspace{5em} - \frac{1}{d-1} \bigg[\frac{1}{N} \sum_{ j=1}^N \frac{1}{h(t)} \int_{t-h(t)}^t \nabla_{x}U( x_t^i- x_s^j)\,ds \bigg]dt
   + A \, d W^i_{t} \bigg).
\end{aligned} 
\end{align}
In the limit $H\to 0$, we obtain a non-retarded interacting particle model with constant speed, whereas, in the limit $H\to\infty$, we obtain the interacting fiber model in (\ref{2dmodelcoordfreeinter}).


\section{Mean-field equation}
\label{meanfieldsection}
The associated mean-field equation for the distribution function $f = f(t,x,\tau)$ may be formally derived from the microscopic equations following
the procedure described, for example, in \cite{BCC,GMR14}. The equation reads
\begin{align}\label{pderot}
  \partial_t f  +  \tau \. \nabla_{x} f + S  f =  L  f
\end{align}
with deterministic force term $S f = S^V\!f + S^U\!f$, given by
\begin{align}\label{Srot}
 \begin{aligned}
  S^V\! f &= - \nabla_{\tau} \. \left(f\left(I- \tau \otimes \tau\right) \frac{1}{d-1} \nabla_x V\right),\\
  S^U\!f &= - \nabla_{\tau} \. \left(f\left(I- \tau \otimes \tau\right) \frac{1}{d-1} \frac{1}{h(t)}\int_{t-h(t)}^t \int_{\R^d} \nabla_x U(x-y)  \rho (s,y)\,dy\,ds\right),
 \end{aligned}
\end{align}
 where $\nabla_{\tau}$ is the gradient on  $\mathbb{S}^{n-1}$, and diffusion operator
\begin{align}\label{Lrot}
   L f =\frac {A^2}{2} \Delta_\tau f,
\end{align}
where $\Delta_\tau $ denotes the Laplace--Beltrami operator on $\S^{d-1}$. In addition, we have the normalization $\int_{\R^d} \rho\,dx =1$, where $\rho$ denotes the zeroth-moment $\rho = \int_{\S^{d-1}} f\, d \nu$.


\begin{remark}
A  stationary solution of the mean-field equation (\ref{pderot}) 
is given by the time-independent solution of
\[
  \tau \. \nabla_{x} f +  Sf = \frac {A^2}{2} \Delta_\tau f.
\]
Looking for a solution independent of $\tau$, we get
\[
 \rho\, \nabla_x\Big( \ln\rho + V + U\star\rho\Big) = 0.   
\]
This leads to the integral equation
\begin{align}\label{integral}
\ln \rho   + V + U \star \rho  = c,
\end{align}
where the constant $c$ is fixed by the normalization $\int_{\R^d}\rho\,dx=1$. The integral equation may be written in the equivalent fixed-point form
\begin{align}\label{eq:fixedpoint}
   \rho = \frac{e^{-V - U \star \rho}}{\int_{\R^d} e^{-V - U \star \rho}\, dx}.
\end{align}
On the other hand, the stationary solution may also be characterised as the (unique) minimizer of the 'energy functional'
\[
 \mathcal{F}(\rho) := \int_{\R^d} (\ln \rho - 1)\rho\,dx + \int_{\R^d} \big(V + U\star \rho \big)\rho \,dx,
\]
where the first term describes the internal energy (entropy), and the second describes the potential energy. In contrast  to the case without interaction, the stationary solution has to be determined numerically.
\end{remark}

The rest of this section is devoted to a rigorous proof of the mean-field limit for the deterministic interacting fiber model 
\begin{align}\label{eq:deterministic}
\begin{aligned}
 \frac{d x_t^i}{dt}&=\tau_t^i\\
 \frac{d\tau_t^i}{dt}&= \big(I - \tau_t^i \otimes  \tau_t^i\big) \circ \bigg( - \frac{1}{d-1} \nabla_{x} V( x_t^i ) \\
    &\hspace{10em} - \frac{1}{d-1} \frac{1}{N} \sum_{ j=1}^N \frac{1}{h(t)} \int_{t-h(t)}^t \nabla_{x}U( x_t^i- x_s^j)\,ds \bigg),
\end{aligned} 
\end{align}
towards the Vlasov type equation
\begin{align}\label{eq:fibervlasov}
  \partial_t f  + \tau \. \nabla_{x} f + S  f =  0,
\end{align}
where $S=S^V + S^U$ is as given in (\ref{Srot}).

\subsection{Mean-field limit of a retarded system}
We consider a system of $N\in\N$ interacting particles with state $Z^i_t\in\R^m$, $m\in\N$ at time $t\in\R_+$
\begin{equation}\label{eq:general_ode}
 \frac{d Z^i_t}{d{t}} = a(Z^i_t) + \frac{1}{N}\sum_{j=1}^N \frac{1}{h(t)}\int_{t-h(t)}^t B(Z^i_t,Z^j_s)\,d{s},\quad Z^i_0=z^i\in\R^m
\end{equation}
for each $i\in\{1,\ldots,N\}$, where as above
\[
 h(t) = \begin{cases} t & \text{for }t\le H \\ H & \text{for }t>H\end{cases},\quad H\in(0,\infty),\quad\text{or}\quad h(t)=t,
\]
and $a\in\Lip(\R^m)$, $B\in\Lip_b(\R^m\times \R^m;\R^m)$ be globally Lipschitz, satisfying
\begin{align*}
 \left\lbrace 
 \begin{aligned}
  \sup_{\hat z\in\R^m} |B(z_1,\hat{z}) - B(z_2,\hat{z})| & \le \Lip(B)\,1\wedge|z_1-z_2|,\\
  \sup_{z\in\R^m} |B(z,\hat{z}_1) - B(z,\hat{z}_2)| & \le \Lip(B)\,1\wedge |\hat z_1- \hat z_2|,
 \end{aligned}\right.
\end{align*}
where we use the notation $x\wedge y = \min\{x,y\}$.
As in (\ref{2dmodelcoordfreeinter}) and (\ref{modelshortdelay}), (\ref{eq:general_ode}) is a retarded system of ordinary differential equations. 

Denote by $\P_r(\R^m)$, $r>0$, the set of Borel probability measures $\mu$ such that
\[
 \int_{\R^m} |z|^r\,\mu(dz) < \infty.
\]
For every probability measure $\mu\in\P_1(\R^m)$, we set
\begin{equation}\label{eq:Kmu}
 \K\mu(z) := \mu(B(z,\.))=\int_{\R^m}B(z,\hat{z})\,\mu(d\hat{z}),\quad z\in\R^m.
\end{equation}
Taking the empirical measure $\mu^N_s = \frac{1}{N}\sum_{j=1}^N \delta(\.-Z^j_s)\in \P_1(\R^m)$ as $\mu$ in (\ref{eq:Kmu}) gives
\[
 \K{\mu^N_s}(Z^i_t) = \int_{\R^m}B(Z^i_t,\hat{x})\,\mu_s^N(d{\hat{x}}) = \frac{1}{N}\sum_{j=1}^N B(Z^i_t,Z^j_s).
\]
Therefore, we may consider a mean-field equation given by
\begin{align}\label{eq:meanfieldode} 
  \frac{d Z_t}{dt}(z) = F(t,Z_t(z),\{\mu\}_t):=a(Z_t(z)) +  \frac{1}{h(t)}\int_{t-h(t)}^t \K{\mu_s}(Z_t(z))\,d{s},
\end{align}
with initial condition $Z_0(z)=z$, $\text{law\,}(Z_0)=\mu_0$. Here $\mu_t = Z_t(\.)_\# \mu_0$ denotes the push forward of the measure $\mu_0\in \P_1(\R^m)$, i.e., 
\[
 \int_{\R^m} \varphi(z)\,\mu_t(dz) = \int_{\R^m} \varphi(Z_t(z))\,\mu_0(dz)\qquad\forall\varphi\in\C_b(\R^m),
\]
under the flow $Z\in\C(\R_+\times\R^m;\R^m)$ generated by the mean-field equation, and $\{\mu\}_t$ denotes the family of measures $\{\mu_s,s\in[0,t]\}$ up to time $t>0$.

We begin with an existence and uniqueness result for solutions of (\ref{eq:meanfieldode}). Our results generalizes those given in \cite{G12}, which follows the general scheme introduced in \cite{dobru} (cf.~\cite{BH,neunzert}). The proof relies on a slight modification of the proof in \cite[Proposition~4.1]{G12}. For the sake of completeness, we include the proof in Appendix~A.

\begin{proposition}\label{prop:exists}
 Let the assumptions on $a$ and $B$ above be satisfied. Then the mean-field equation (\ref{eq:meanfieldode}) admits a unique global solution $Z\in\C(\R_+\times\R^m;\R^m)$.
\end{proposition}

\begin{remark}
 The family of measures $\{\mu_t=Z_t(\.)_\#\mu_0\}\subset\P_1(\R^m)$ generated by the flow $Z\in\C(\R_+\times\R^m;\R^m)$ provides a weak solution to the Vlasov equation
 \begin{align}\label{eq:vlasov}
  \partial_t\mu_t + \text{div\,}(F(t,z,\{\mu\}_t)\mu_t)=0,
 \end{align}
 More precisely, for all $h\in\C_0^\infty(\R^m)$, the functions $\mu_t(h)$ are differentiable,
 \[
  \frac{d\mu_t(h)}{dt} = \mu_t(F(t,\.,\{\mu\}_t)\.\nabla h)
 \]
 for all $t>0$, and $\mu_t(h)\to \mu_0(h)$ for $t\to 0_+$.
\end{remark}

Next, we show that solutions to the mean-field equation (\ref{eq:meanfieldode}) depend continuously on the initial probability measures $\mu_0\in \P_1(\R^m)$. To do so, we need to measure the difference of two probability measures $\mu,\nu\in\P_1(\R^m)$. A convenient way is to use the Monge--Kantorovich--Rubinstein distance $\dist$ on $\P_1(\R^m)$ defined in \cite{KantoRubin1958} (cf.~\cite{Villani2008}),
\[
 \dist(\mu,\nu) = \inf_{\pi\in\Pi(\mu,\nu)}\iint_{\R^m\times\R^m} 1\wedge|x-y|\,\pi(dxdy),
\]
where $\Pi(\mu,\nu)$ is the set of Borel probability measures on $\R^m\times\R^m$ such that
\[
 \iint_{\R^m\times\R^m} (\phi(x) + \psi(y))\pi(dx\,dy) = \int_{\R^m}\phi(z)\,dz + \int_{\R^m}\psi(z)\,dz
\]
for all $\phi,\psi\in\C_b(\R^m)$. This distance is also called the Wasserstein distance.

\begin{proposition}\label{prop:stability}
 Let $\mu_0^j\in\P_1(\R^m)$, and $Z^j\in\C(\R_+\times\R^m;\R^m)$ be the corresponding solution to the mean-field equation (\ref{eq:meanfieldode}) with initial conditions
 \[
  Z_0^j(z) = z,\quad \text{law}\,(Z_0^j) = \mu_0^j,\quad \text{for}\ j=1,2
 \]
 then for every $T>0$, the following stability estimate holds
 \[
  \sup_{t\in[0,T]}\dist(\mu_t^1,\mu_t^2) \le c\,(1+T)e^{cT}\,\dist(\mu_0^1,\mu_0^2),
 \]
 for some constant $c=c(F,H)>0$, where $\mu_t^j := Z_t^j(\.)_\#\mu_0^j$
\end{proposition}
\begin{proof}
 Let $\pi_0\in \Pi(\mu_0^1,\mu_0^2)$, and define
 \[
  \D(t):= \iint_{\R^m\times \R^m} 1\wedge |Z_t^1(z_1)-Z_t^2(z_2)|\,\pi_0(dz_1 dz_2) 
 \]
 Following the arguments in the proof of Proposition~\ref{prop:exists}, we derive the estimate
 \[
  \D(t) \le \D(0) + \Lip(F)\int_0^t \D(s)\,ds + \Lip(B)\int_0^t \frac{1}{h(s)}\int_{s-h(s)}^s \D(\sigma)\,d\sigma\,ds.
 \]
 We only show the estimate for the case 
 $H=\infty$, i.e. $h(t) =t$. The general 
   case $H>0$ may be shown analogously. Similar to Proposition~\ref{prop:exists}, by simple but tedious computations, we have
 \[
 \D(t) \le \D(0) + \Lip(F)\int_0^t g(t,s)\,\D(s)\,ds,
 \]
 with $g(t,s)= 1 +\ln(t) - \ln(s)$. Recursively, we obtain
 \[
  \left(1 - \Lip(F)^k\frac{(n+1)T^k}{k!} \right)\sup_{t\in[0,T]}\D(t)\ \le \ \sum_{\ell=0}^k \Lip(F)^\ell\frac{(\ell+1)T^\ell}{\ell!}  \D(0).
 \]
 Therefore, passing to the limit $k\to\infty$ yields
 \[
  \sup_{t\in[0,T]}\D(t) \le c(1+T)e^{cT}\,\D(0),
 \]
 with $c=\max\{1,\Lip(F)\}$. Since
 \[
  \D(t) = \iint_{\R^m\times\R^m} 1\wedge|z_1-z_2|\pi_t(dz_1dz_2),
 \]
 where $\pi_t$ is the push forward measure of $\pi_0$ under the map $(Z_t^1,Z_t^2)$, and
 \[
  \pi_0\in\Pi(\mu_0^1,\mu_0^2)\ \Longrightarrow \ \pi_t\in \Pi(\mu_t^1,\mu_t^2)
 \]
 for any $t\ge 0$, we have
 \[
  \sup_{t\in[0,T]}\dist(\mu_t^1,\mu_t^2) \le c(1+T)e^{cT}\,\D(0)
 \]
 Finally, optimizing in $\pi_0$ yields the required assertion.
\end{proof}

\begin{remark}
For any $N\in \N$, the family of empirical measures $\{\mu_t^N,t\ge 0\}\subset\P_1(\R^m)$ defined by
$\mu_t^N=Z_t(\.)_\#\mu_0^N\in\P_1(\R^m)$, where 
\[
 \mu_0^N = \frac{1}{N}\sum_{j=1}^N \delta(\.-z^j),\quad z^j\in\R^m,\quad j=1,\ldots,N
\]
is a weak solution to the Vlasov equation (\ref{eq:vlasov}) by construction. Therefore, if we know that $\lim_{N\to\infty}W_1(\mu_0,\mu_0^N)\to 0$ for some $\mu_0\in\P_1(\R^m)$, then the stability result given in Proposition~\ref{prop:stability} provides the convergence
\[
 \lim_{N\to\infty}W_1(\mu_t,\mu_t^N)\to 0\quad\text{for any}\ t\in[0,T],
\]
where $\mu_t=Z_t(\.)_\#\mu_0$, with $Z\in\C([0,T]\times\R^m;\R^m)$.
\end{remark}

\subsection{Application to the retarded fiber equations} 
We now use the results above to show the mean-field limit of (\ref{eq:deterministic}) towards (\ref{eq:fibervlasov}).
For this reason, we denote $Z_t^i=(x_t^i,\tau_t^i)\in\R^d\times\R^d$, and write $a=(a_1,a_2)$, $B=(B_1,B_2)$ with
\begin{align*}
 a_1(Z_t^i) = \tau_t^i ,&\quad
 a_2(Z_t^i) = - \frac{1}{d-1} \big(I - \tau_t^i \otimes  \tau_t^i\big)\circ\nabla_{x} V( x_t^i), \\
 B_1(Z_t^i,Z_s^j) = 0,&\quad B_2(Z_t^i,Z_s^j) = - \frac{1}{d-1}\big(I - \tau_t^i \otimes  \tau_t^i\big)\circ\nabla_{x} U( x_t^i- x_s^j ).
\end{align*}
Obviously, we are unable to directly apply the results developed above, since $a$ and $B$ do not satisfy the assumptions above. Consider $B_2$ for the moment. Tedious but simple computations yield
\begin{align*}
 |B_2(z_1,\hat z) - B_2(z_2,\hat z)| &\le \Lip(\nabla_x U)\Big( (1+|\tau_1|^2)|x_1-x_2| + (|\tau_1| + |\tau_2|)|\tau_1-\tau_2|\Big) \\
 |B_2(z,\hat z_1) - B_2(z,\hat z_2)| &\le \Lip(\nabla_x U)(1+|\tau|^2)|\hat x_1-\hat x_2|
\end{align*}
Hence, $B$ is not globally Lipschitz. Fortunately, if we only consider $B$ on the subset $\M\subset \R^d\times\R^d$, then $\Lip(B)=2\Lip(\nabla_x U)$. Consequently, $B$ is globally Lipschitz on $\M$. Clearly, the same conclusion holds for $a$.

In order to ensure that $Z_t^i\in\M$ for all $t\ge 0$, we observe that
\[
 \frac{d}{dt}\frac{1}{2}|\tau_t^i|^2 = 0\quad\text{for all}\ t\ge 0,\quad i=1,\ldots,N.
\]
Therefore, if we start with $\tau^i\in \S^{d-1}$, then also $\tau_t^i\in \S^{d-1}$ for all other times $t>0$, and we may apply our results obtained above to initial measures $\mu_0\in\P_1(\R^d\times\R^d)$ with $\text{supp}(\mu_0)\subset \M$. Indeed, since $Z_t(x,\tau)\in \M$ for any $(x,\tau)\in\M$, we may consider the flow on $\M$ and the push forward $\mu_t = Z_t(\.)_\#\mu_0\in \P_1(\M)$, as soon as $\text{supp}(\mu_0)\subset \M$.

\section{Large diffusion scaling}\label{diffusionsection}
In this section we formally investigate situations with large  values for  the noise amplitude on a diffusive time scale, i.e., we change $\tilde{L} = \epsilon L$ and $\tilde{t} = \epsilon t$ (cf.~\cite{BLP78,HKMO09}). Replacing the scaled operators in (\ref{pderot}) and omitting the tilde lead to the scaled equation
\begin{align}\label{pdescaled2}
 \epsilon\,\partial_t f  + \tau \. \nabla_x f +   Sf  =   \frac{1}{\epsilon} L f.
\end{align}
We use a {\em Hilbert expansion} of the form $f = f_0 + \epsilon f_1 + \dots$ for (\ref{pdescaled2}).
To order 1, we simply get $ f_0 = f_0(x) = \rho(x)$. To order $\epsilon$, one obtains
\[
 \tau \. \nabla_{x} f _0   + S f_0 = \frac{A^2}{2} \Delta_\tau f_1,
\]
which, due to the Fredholm alternative, gives
\[
 f_1 = - \frac{2}{A^2(d-1)}  \tau  \. f_0\, \nabla_x \left(  \ln f_0 + V + \frac{1}{h(t)} \int_{t-h(t)}^t (U\star f_0)(s,\.)\,ds \right).
\]
Integrating (\ref{pdescaled2}) with respect to $d\nu$ gives
\[
 \epsilon\,\partial_t \int_{\S^{d-1}} f\, d\nu + \nabla_{x} \.  \int_{\S^{d-1}} \tau  f \,d\nu =0. 
\]
Considering terms up to order $\epsilon$, we obtain
\[
 \partial_t f_0 + \nabla_{x} \.  \int_{\S^{d-1}} \tau  f_1\, d \nu = 0.
\]
Therefore, inserting $f_1$ and computing the integral over the tensor product yields
\[
  \partial_t \rho = \frac{2}{d(d-1)A^2}  \nabla_{x} \. \left[\rho\,\nabla_x\left( \ln \rho   + V + \frac{1}{h(t)} \int_{t-h(t)}^t (U\star\rho)(s,\.)\,ds \right)\right].
\]
This equation is similar to an equation derived in \cite{BGKMW07,HKMO09,KST14}  for the case without an interaction potential. 
The stationary equation reads
\[
 \rho\,\nabla_x\Big( \ln \rho   + V + U\star\rho \Big)  =0,\quad \int_{\R^d} \rho\, dx =1,
\]
 which leads again to the integral equation (\ref{eq:integral})
as for the mean-field case.

\section{Numerical Method and Results}\label{numerics}
In this section we investigate the qualitative behavior of solutions corresponding to the interacting fiber model numerically. More specifically, we consider the case of isotropic fibers in three spatial dimensions.

\subsection{Numerical methods}

We describe numerical solvers for the microscopic, mean-field and macroscopic equations, respectively.

\subsubsection{Microscopic model}
To solve equations (\ref{2dmodelcoordfreeinter}) or (\ref{modelshortdelay}) numerically we use the Euler--Maruyama method.
Using  It\^o integration  (\ref{2dmodelcoordfreeinter}) is written as
\begin{align*}
dx_t^i &= \tau_t^i \,dt \\
d\tau_t^i &= -\frac{1}{2}\big(I-\tau_t^i\otimes\tau_t^i\big)\bigg(\nabla_{x} V(x_t^i) +\frac{1}{N}\sum_{j=1}^N\frac{1}{t}\int_0^t\nabla_{x} U(x_t^i-x_s^j)\,ds\bigg)dt\\
&\hspace*{18em} + A^2\tau_t^i\,dt +A\big(I-\tau_t^i\otimes\tau_t^i\big)dW_t^i.
\end{align*}
The It\^o form of (\ref{modelshortdelay}) is obtained in an analogous way. We consider an  equidistant time grid given by $0=t_0<...<t_n$ with step size $\Delta t$. We denote $x^i_n:=x_{t_n}^i$ for the position of fiber $i$ at time $t_n=n\dt$. The time integration in (\ref{2dmodelcoordfreeinter}) is approximated by
\begin{align}\label{nobuf}
 \frac{1}{t}\int_0^t\nabla_{x}U(x_t^i-x_s^j)\,ds\ \sim\ \sum_{k=1}^n\frac{1}{t_n}\nabla_{x}U(x^i_n-x^j_k)\dt = \frac{1}{n}\sum_{k=1}^n \nabla_{x}U(x^i_n-x^j_k).
\end{align}
For (\ref{modelshortdelay}) with $H>0$, we set up a buffer to store the previous positions $x_t^i$. The buffer size is determined by $n_{buf}=h(t)/\dt$. The time integration is then  approximated by
\begin{align}\label{buf}
\frac{1}{h(t)}\int_{t-h(t)}^t\nabla_{x} U(x_t^i-x_s^j)\,ds\ \sim\ \sum_{k=1}^{n_{buf}}\frac{1}{h(t_n)}\nabla_{x} U(x^i_n-x^j_{n-k})\dt.
\end{align}
Hence, the Euler--Maruyama iteration is given by
\begin{align*}
x^i_{n+1}= x^i_n &+\dt\. \tau^i_n\\
\tau^i_{n+1}= \tau^i_n &+\dt\. \Big[-\frac{1}{2}\big(I-\tau^i_n\otimes\tau^i_n\big)\Big(\nabla_{x}V(x^i_n)\\
&+\frac{1}{N}\sum_{j=1}^N\Big(\sum_{k=1}^n\frac{1}{t_n}\nabla_{x}U(x^i_n-x^j_k)\dt\Big)\Big)-A^2\tau^i_n\Big]\\
&+\sqrt{\dt}\. A\big(I-\tau^i_n\otimes\tau^i_n\big)R^i_n
\end{align*}
where $R^i_n\in\R^3$, for $i=1,...,N$, is a vector containing three normally distributed pseudorandom values. For (\ref{modelshortdelay}), one has to replace the sum over all previous time steps (\ref{nobuf}) by the sum over the time step buffer (\ref{buf}). Typical simulations of the microscopic system include $N\sim 10^6$ fibers, whose positions are used to obtain a histogram, which will be compared to the solution of the mean field equation described below. For such large values of $N$, the evaluation of the interaction term is quite costly. Several measures were taken to ensure reasonable computing time: 
\begin{enumerate}
 \item We did not sum up every time step in (\ref{nobuf}) and considered only every $\tilde{n}$-th time step. In that way, the results are not altered in a significant way, if $\tilde{n}$ is not too large.
 \item Furthermore, the Fortran code carrying out the microscopic simulations was parallelized using OpenMP. If a reasonable scaling of the computing time with the number of processors is to be achieved, the parts of the code running in parallel have to be as mutually independent as possible, in order to reduce the overhead in communication between parallel threads.
 \item Therefore, we divided $N$ into smaller groups of fibers, which then interact only within each group, thereby enabling a parallel computation with very little overhead. The histogram for the spatial density is then produced from the position data of all the groups. 
\end{enumerate}

\subsubsection{Mean field equation}
The numerical methods used here are an advancement of the schemes used in \cite{J.A.Carrillo2014,RKSZ14,R14}. We apply a second order Strang splitting \cite{C.Cheng1976} to (\ref{pderot}) to obtain subproblems on spatial and velocity domain. We split the equation
\[
\partial_tf+\tau\.\nabla_xf+Sf =Lf,
\] 
into the subequations
\begin{align}
\partial_tf^{(1)}&=-\frac{1}{2}(Sf^{(1)}-Lf^{(1)}), \label{velo1}\\
\partial_tf^{(2)}&=-\tau\.\nabla_xf^{(2)}, \label{spatial}\\
\partial_tf^{(3)}&=-\frac{1}{2}(Sf^{(3)}-Lf^{(3)}). \label{velo2}
\end{align}
Equation (\ref{spatial}) has to be solved only on the spatial domain $\R^3$, while (\ref{velo1}) and (\ref{velo2}) are to be solved on the velocity domain $\S^2$. Note however, that all the computation steps on one grid have to be carried out for every grid point on the other grid.

\begin{figure}
\centering
\includegraphics[clip,width=0.45\columnwidth]{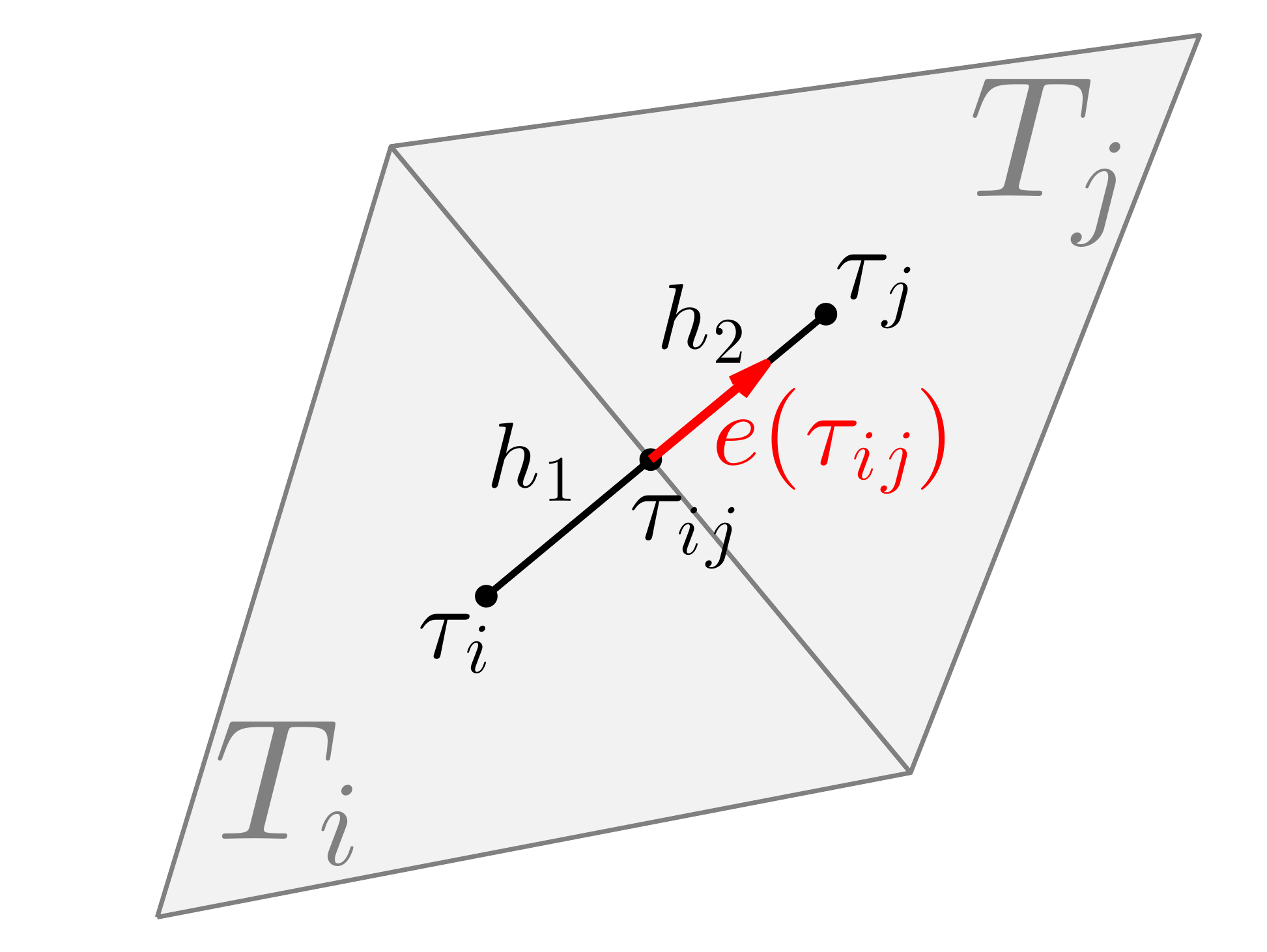}
\includegraphics[clip,width=0.5\columnwidth]{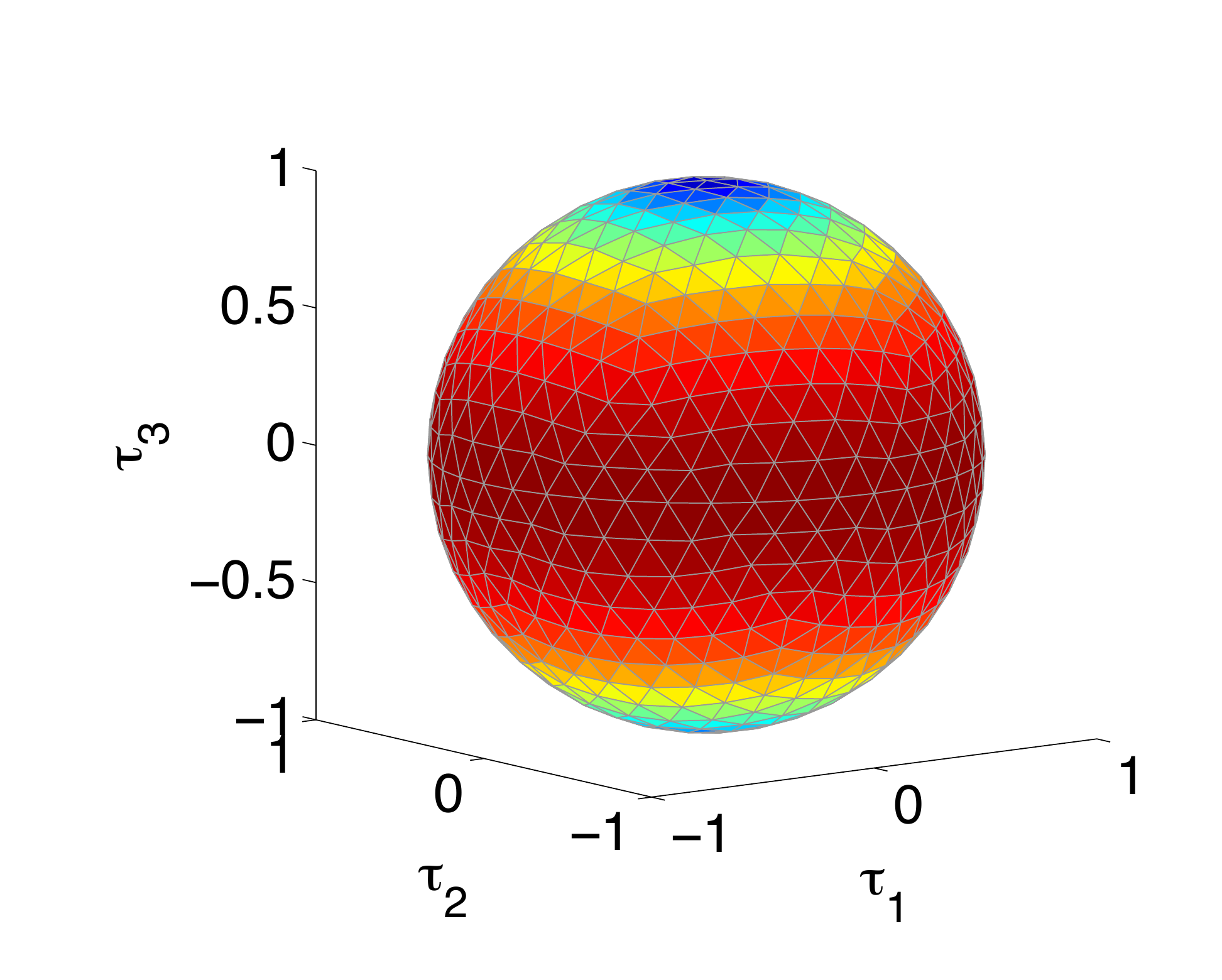}
\caption{The left figure depicts two connected triangles of the geodesic grid, while the right figure shows the full spherical grid.}
\label{gridvis}
\end{figure}  

Firstly, we discuss the solution of (\ref{velo1}) and (\ref{velo2}) with a finite volume discretization of the velocity space based on a geodesic grid (cf.~\cite{RKSZ14,AltGeo3}). The geodesic grid consists of spherical triangles, which are represented with vertices and normals in $\R^3$, neighbor relations and correct distance and surface measures. In this way the geometry is included implicitly into the method. We denote by $T_i$ the $i$-th cell of the grid with cell midpoint $\tau_i$. The cell midpoint is chosen to be the intersecting point of the great circle arcs passing perpendicularly through the midpoints of the cell edges. $|\.|$ denotes a length or surface measure, $T_{ij}$ denotes the edge between cells $T_i$ and $T_j$, $\tau_{ij}$ denotes the edge midpoint, and $e(\tau_{ij})$ denotes the outer normal vector of cell $i$ at the edge midpoint $\tau_{ij}$. The distance from cell midpoint $\tau_i$ to $\tau_j$ is given by $h_{ij}$, which is divided into the parts $h_1$ and $h_2$. Due to the construction of the grid, the distance between cell midpoints does not vary that much, and $h_1,h_2\sim h_{ij}/2$, so all the cells have nearly the same size and shape. The reader is referred to Figure~\ref{gridvis} for a visualization of the grid structure and notations.

As for any finite volume scheme, the solution is obtained via cell averages
\[
 f_i^n=\frac{1}{|T_i|}\int_{T_i}f(t_n,x,\tau)\,d\tau,
\]
where $d\tau$ is the canonical surface measure on the sphere. We define 
\[
 F(t,x,\tau):=\frac{1}{2} (I-\tau\otimes\tau)\Big(\nabla_xV+\frac{1}{h(t)}\int_{t-h(t)}^t\nabla_x U\star\rho\,ds\Big),
\] 
and integrate (\ref{velo1},\,\ref{velo2}) over the cells $[t_n,t\np]\times T_i$ to obtain
\begin{align*}
{f\hnp_i-f_i^n} =\frac{1}{|T_i|}\sum_{j\in N(i)} & \Bigg[\int_{t_n}^{t\np} \int_{T_{ij}}F(s,x,\tau)\. e(\tau) f\,d\tau\,ds\\
&\hspace*{4em} +\frac{A^2}{2}\int_{t^n}^{t\np}\int_{T_{ij}}\nabla_\tau f\. e(\tau)\,d\tau\, ds\Bigg].
\end{align*}
Applying the midpoint rule on the cell edge and in time, one obtains the iteration
\begin{align*}
\frac{f\hnp_i-f_i^n}{\dt} \sim\frac{1}{|T_i|}\sum_{j\in N(i)} & \Bigg[|T_{ij}|F(x,\tau_{ij})\. e(\tau_{ij}) f(t\nph,x,\tau_{ij})\\
& +\frac{A^2}{2}|T_{ij}|\nabla_\tau f(t\nph,x,\tau_{ij})\. e(\tau_{ij})\Bigg].
\end{align*}
The overall order of the method depends on the approximation of $f(t\nph,x,\tau_{ij})$ and the normal flux $\nabla_\tau f(t\nph,x,\tau_{ij})\. e(\tau_{ij})$. In this case, we choose
\begin{gather*}
 f(t\nph,x,\tau_{ij})\sim\frac{h_1+\frac{\dt}{2}(F(x,\tau_{ij})\. e(\tau_{ij}))}{h_{ij}}\. f_i^n+\frac{h_2-\frac{\dt}{2}(F(x,\tau_{ij})\. e(\tau_{ij}))}{h_{ij}}\. f_j^n, \\
 \nabla_\tau f(t\nph,x,\tau_{ij})\. e(\tau_{ij})\sim\frac{f_j^n-f_i^n}{h_{ij}}.
\end{gather*}
The approximation for $f(t\nph,x,\tau_{ij})$ is  similar to the Lax--Wendroff numerical flux function, except that we evaluate $F(x,\tau)\. e(\tau)$ at the edge midpoint $\tau_{ij}$, and not at the cell midpoints $\tau_i$ and $\tau_j$, since  the grid structure only provides normal vectors at the cell interfaces. The value for $\nabla_\tau f(t\nph,x,\tau_{ij})\. e(\tau_{ij})$ is obtained by a finite difference approximation on the connecting circle arc of the cell midpoints $\tau_i$ and $\tau_j$, see \cite{RKSZ14}. Although  the method is not a second order method, the numerical results are close to those of a second order method, see the discussion below and Figure \ref{rate} for the convergence rates of the splitting scheme.

Equation (\ref{spatial}) is solved using a semi-Lagrangian method \cite{KRS07,Sem6} on an equidistant grid $x_{ijk}\in \R^3, i,j,k=1,...,n_x$ with grid size $\dx$.  The characteristic curves $\gamma(t)$ of (\ref{spatial}) starting at grid point $x_{ijk}$ at time $t\np$ are given by
\[
\gamma(t) = x_{ijk}+t\.\tau.
\]
 The derivative of $f$ with respect to time on a characteristic curve is given by
\[
\frac{d}{dt}f(t,\gamma(t),\tau) =\partial_tf(\gamma(t),\tau,t)+\tau\.\nabla_xf(\gamma(t),\tau,t)=0
\]
and it follows, that $f$ is constant along the characteristic curve:
\[
f(t\np,x_{ijk},\tau) = f(t_n,\gamma(-\dt),\tau).
\]
Since only  the values of $f$ at the grid points at time $t_n$ are known,  we have to interpolate $f(t_n,\gamma(-\dt),\tau)$ from $f(t_n,x_{ijk},\tau)$. The order of the method  depends on the order of the interpolation procedure, since the characteristic curves can be computed analytically in this case. An order higher than one  will produce unphysical oscillations at discontinuities of the numerical solution, which has to be prevented. We use a Bezier interpolation \cite{J.A.Carrillo2014}. We use the notations  $x_{ijk}=(x_{ijk}^l)_{l=1,2,3}$,  $\gamma(-\dt)=(\gamma_l(-\dt))_{l=1,2,3}\in[x_{ijk}^1,x_{ijk}^1+\dx]\times[x_{ijk}^2,x_{ijk}^2+\dx]\times[x_{ijk}^3,x_{ijk}^3+\dx]$ and define $\xi=(\xi_l)_{l=1,2,3}$ as follows
\[
\xi_l=\frac{\gamma_l(-\dt)-x^l_{ijk}}{\dx},\quad l=1,2,3.
\]
Then the interpolated values on each cell are given by an interpolating polynomial of three space variables of the following form:
\[
f(t_n,\gamma(-\dt),\tau) = \sum_{\lambda,\mu,\nu=0}^3B_{\lambda,3}(\xi_1)B_{\mu,3}(\xi_2)B_{\nu,3}(\xi_3)v_{\lambda\mu\nu},
\]
where $B_{l,3}(\xi)=\binom{l}{3}\xi^l(1-\xi)^{3-l}$, $l\in\{0,1,2,3\}$ are the cubic Bernstein polynomials, and $v_{\lambda\mu\nu}$ are the control values for the interpolation. This interpolant never leaves the convex hull of the control values. These values are chosen appropriately, such that the required interpolation order is preserved on smooth sections of the solution, and oscillations are prevented elsewhere. This can be done by computing an interpolating Newton polynomial of appropriate order, performing a change of basis to the Bernstein polynomials and then shifting control points back into the convex hull of the neighboring grid function values. See \cite{J.A.Carrillo2014} for more information on this method

We note that semi-Lagrangian methods are not conservative in general. For linear advection, we get a conservative method using the Bezier interpolation procedure without slope limiting. In \cite{Sem3}, this is shown for a third order interpolating polynomial, which is identical to the Newton polynomial. Since we reproduce this polynomial in the Bernstein polynomial basis, the same computation can be done for the Bezier interpolation used here. However, the slope limiting procedure destroys mass conservation and we have to apply conservation techniques  as described  in \cite{Sem2,KRS07}.


\subsection{Stationary equation}
The stationary equation (\ref{integral}) is solved via an iteration scheme. We use the fixed-point form (\ref{eq:fixedpoint}) and the iterative scheme
\begin{align*}
\rho_{n+1} = \frac{e^{-\left( V + U \star \rho_n\right)}}{\int e^{-\left( V + U \star \rho_n\right)} dx}.
\end{align*}
As starting point $\rho_0$ we use the solution of the stationary equation without interaction, namely $\rho_0=Ce^{-V}$. The iteration is well-defined, i.e., $\rho_{n}$ is integrable and has integral one. In addition,  $\rho_n$ is strictly positive for all $n \in \N$.

For the implementation we have to approximate the convolution $U \star \rho$ or $\nabla_xU\star\rho$. This is done using the midpoint rule on each grid cell. One obtains
\begin{align}
(U \star \rho) (x_{ijk})= \int U(x_{ijk}-y) \rho(y) dy \ \approx\ \sum_{\lambda,\mu,\nu} U(x_{ijk}-x_{\lambda\mu\nu}) \rho(x_{\lambda\mu\nu}) (\dx)^3, 
\end{align}
where $x_{ijk}$ are the grid points and $\dx$ is the (constant) grid size in each direction, as before. For the computation of the convolution, one has to compute the distance matrix $(|x_{ijk}-x_{\lambda\mu\nu}|)$, which consumes a lot of computing time and memory. One advantage over the solution of the microscopic system is, that the grid points are fixed in contrast to the fiber positions in the microscopic context. Therefore, the distance matrix may be precomputed. Furthermore, one can precompute for each grid point a list of neighbor points, which have a relevant effect on the computation for a given set of parameters and the purely repulsive potential (\ref{Uinteraction3}). Only these points provide a relevant contribution to the interaction force.

The current implementation only considers interaction forces, which are higher than $0.1$ percent of the maximal occurring force as relevant. A considerable amount of memory can be saved by applying this weak restriction. The time integration in (\ref{Srot}) for $h(t)=t$ does not increase the effort in the mean-field context, because it can be realized by just summing up the interaction forces from different time steps. For the case $H\in(0,\infty)$, a time step buffer as in the microscopic context is used.


\subsection{Numerical results}

We show numerical results for the microscopic model (\ref{modelshortdelay}) and the mean-field equation (\ref{pderot}) in three spatial dimensions. In particular, we investigate  the convergence of the microscopic and mean-field solutions to the corresponding stationary state. The stationary states obtained from microscopic and mean-field computations are compared to the solution of (\ref{eq:fixedpoint}). For a more detailed investigation of the convergence to equilibrium for different values of the delay $H$, we investigate the time decay of the distance to equilibrium for microscopic and mean-field equations and compare it  to the one obtained for the case without interaction. Moreover, the sensitivity of the stationary solutions with respect to the parameters of the interaction model is investigated.

For all computations we consider the three dimensional case for (\ref{modelshortdelay}). Unless otherwise stated, we choose the interaction potential given in (\ref{Uinteraction3})  with $k=10$, $C=10$ and $R = 1.4$. The coiling potential $V$ is chosen as $V(x) = | x |^2/2$.

We are especially interested in the comparison of the numerical results for the mean-field equation and the microscopic stochastic model. 
For this comparison, we consider different noise coefficients $A$ and a box-shaped initial condition for the mean-field equation given by
\begin{align*}
f_0(x,\tau)=C\.\begin{cases}
1, & \text{for}\ x\in[-1,1]^3,\, \tau_3>0\\
0, & \text{else}
\end{cases},
\end{align*}
where $\tau_3$ denotes the third component of the vector $\tau$ and $C$ is a normalizing constant. The corresponding initial condition for the microscopic system can be obtained by randomly placing particles in the cube $[-1,1]^3$ with initial velocities having the third component larger than zero.  We start the investigation by studying order of convergence of the numerical method for the mean-field equation described above.

\subsubsection{Order of convergence of the numerical method for the mean-field equations without interaction}

In this section we give a numerical investigation of the method for the mean field equation for the case without interaction. Although the method is not a second order method, owing to the geodesic grid, the numerical results show that the numerical order of convergence is of order two. Let $n_x$ be the number of  spatial grid points in each direction and $n_k$  the number of cells on the sphere, $h_x$ and $h_k$ are the respective grid sizes. Note, that $h_k$ is the average grid size on the geodesic grid. The spatial step size was chosen to match $h_k$.
The following grid sizes were simulated:
\begin{center}
\begin{tabular}{|c|c|c|c|}\hline
$n_x$ & 11& 21& 41 \\\hline
$n_k$ & 20& 80& 320 \\\hline
$h_x$ & 0.76 & 0.38 & 0.19 \\\hline
$h_k$ & 0.730 & 0.353 & 0.175 \\\hline
$err_{L2}$ & 0.004604 & 0.000777 & 0.000175 \\\hline
\end{tabular}\\[1em]
\end{center}
From this data, we obtain the convergence plot shown in Figure~\ref{rate}.

\begin{figure}
\centering
\includegraphics[clip,width=0.5\textwidth]{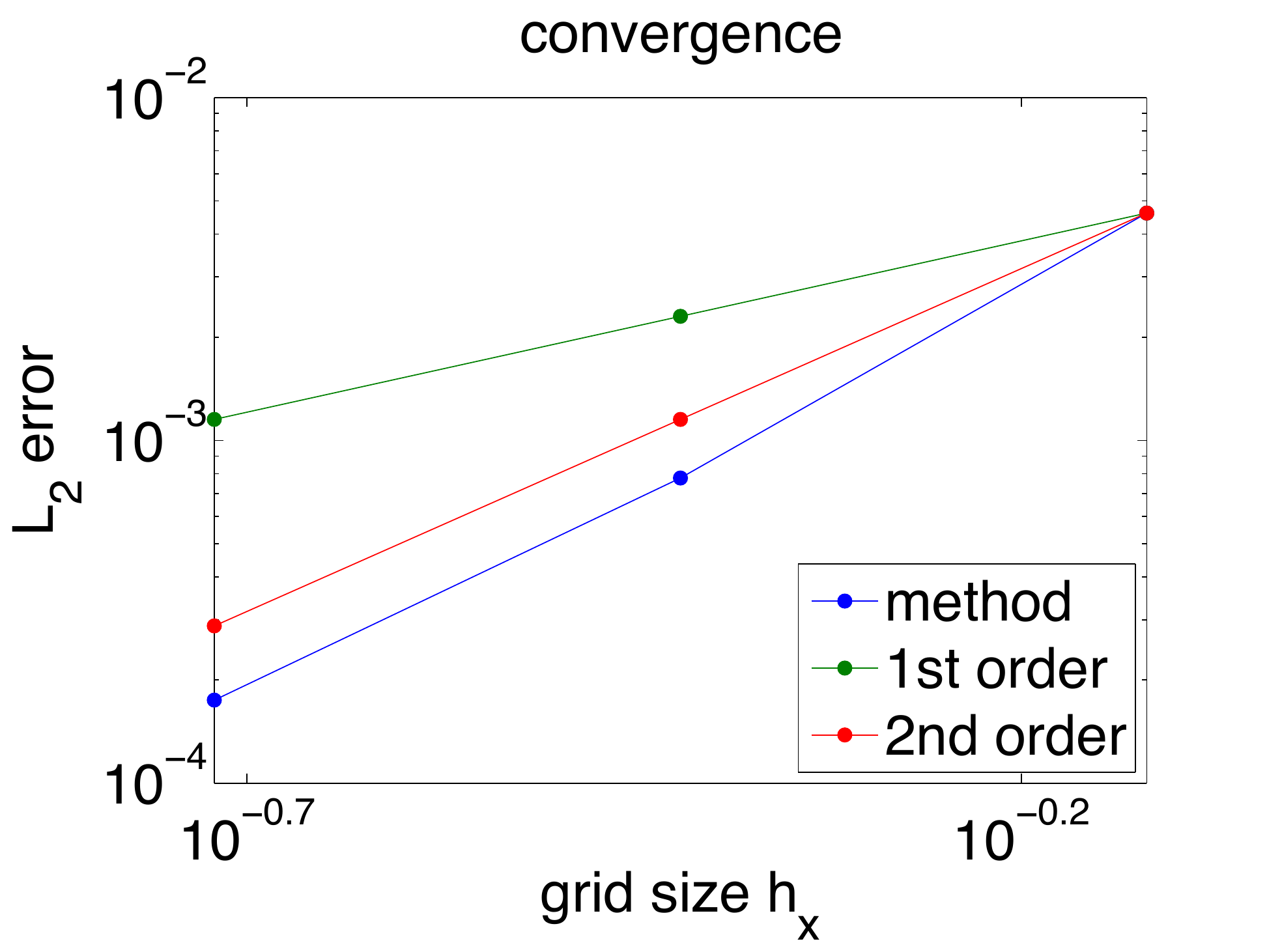}
\caption{Order of convergence for the mean-field equation without interaction. The green and red lines correspond to first and second order error, respectively.}
\label{rate}
\end{figure}  

\subsubsection{Stationary solution}

In Figure~\ref{figstat}, we consider the stationary distributions obtained from the microscopic and mean-field equations for different values of the delay $H$. We show radial plots of the radially symmetric density function. The case without interaction is compared to the case with interaction, with delay given by  $H\in\{0,0.1,0.5,\infty\}$. The grid resolutions simulated for the mean-field equations were $n_x=40$ points in each spatial direction with a grid size of $h_x=0.205$  for the cases without interaction and the cases $H\in\{0,\infty\}$. To slightly reduce the computing time, we used a smaller spatial resolution of $n_x=35$ and $h_x=0.235$ in the cases with $H\in\{0.1,0.5\}$. In all cases, the sphere was discretized by 320 geodesic triangles, which led to an average cell midpoint distance of $0.175$. These grid sizes where matched by the histograms generated from the microscopic fiber positions and velocities. 

To achieve smooth data, the total number of simulated fibers has to be chosen sufficiently large. In the case without interaction, the microscopic density is generated from $4.8\.10^6$ simulated fiber positions. For the cases with interaction, we computed 800 realizations of a system of 600 interacting fibres and generated the histogram from the fibre positions of all the realizations, which means that the histogram is based on $4.8\. 10^5$ microscopic fiber positions. As mentioned earlier, the interaction computation in the microscopic context is quite costly. So, although the computing time for $4.8\.10^6$ fibres without interaction took only a couple of minutes, depending on how many processing cores we have at our disposal, it took up to several days of computing time for the case $H=0.5$, which is the most expensive task here. Computing times for the mean-field solver lie in the range of $3$ to $24$ hours on the given hardware, which was an Intel XEON E5 2670 with 8 cores at 3.3GHz.

As expected, the stationary distribution does not change with different values of $H$. However the stationary distribution for the case with interaction is significantly wider than the distribution for the case without interaction. We note that the solutions obtained by the microscopic and by the mean-field equation are in very good agreement for the cases considered here.

\begin{figure}
\includegraphics[clip,width=0.5\textwidth]{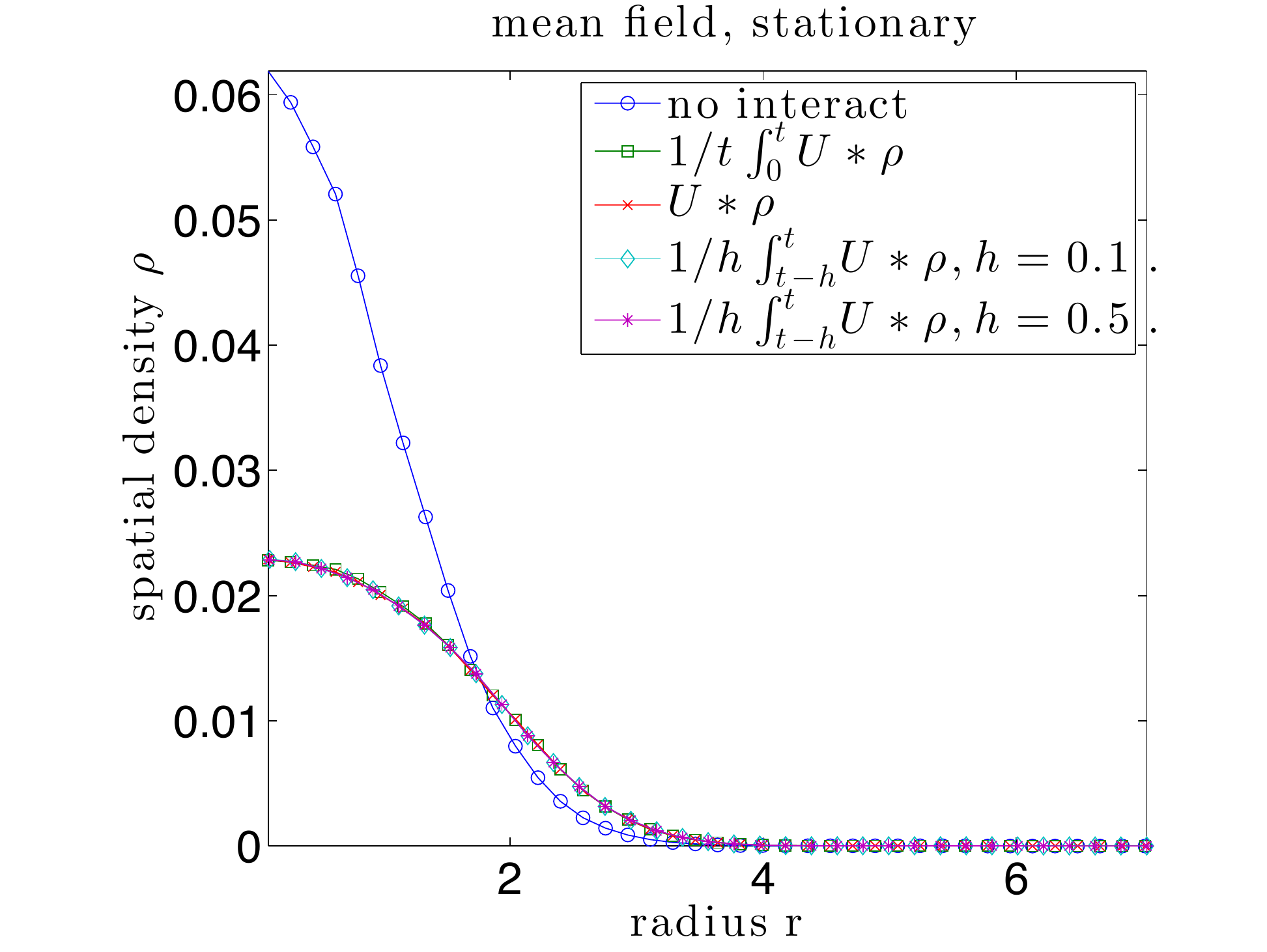}
\includegraphics[clip,width=0.5\textwidth]{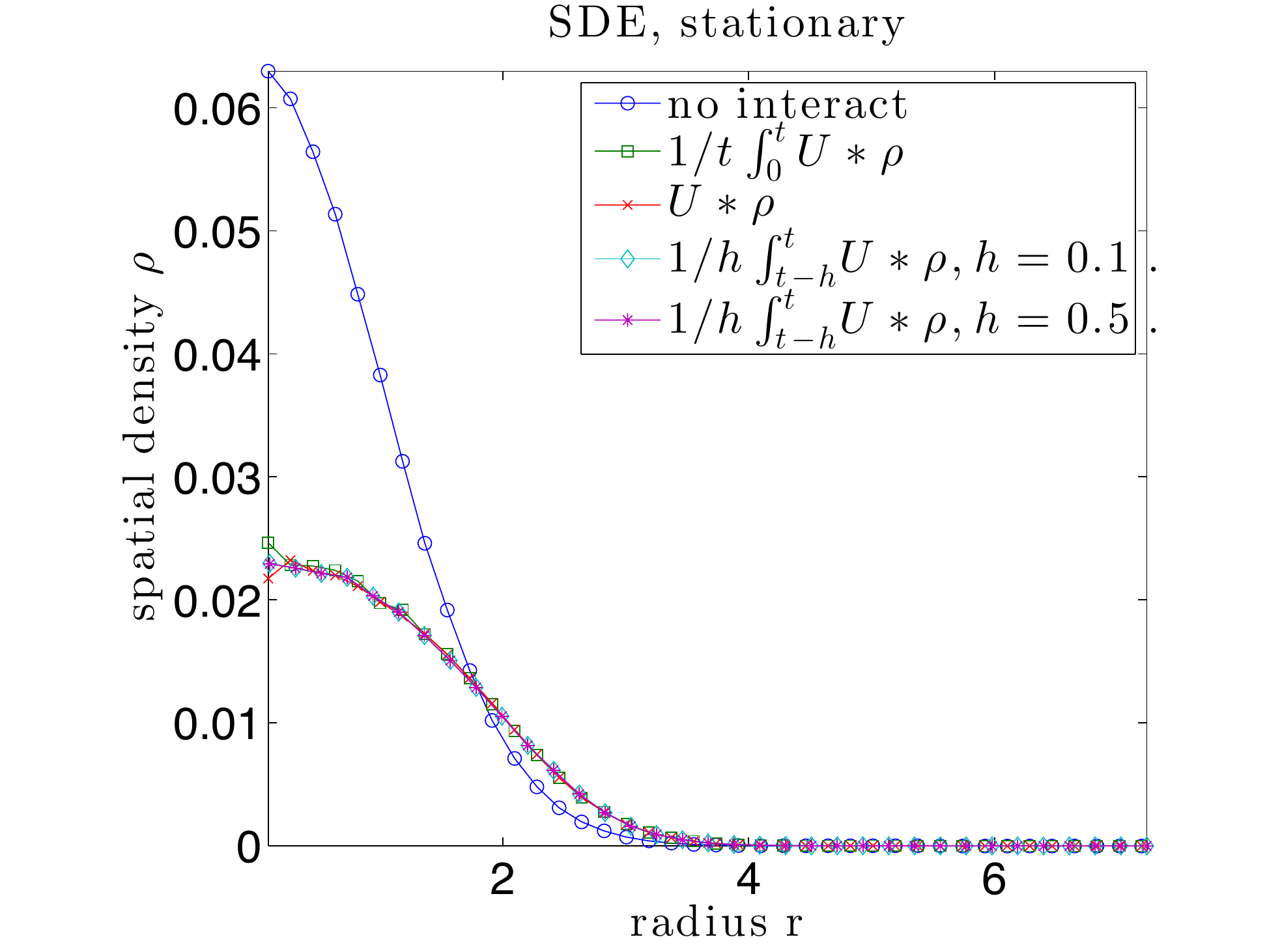}
\caption{In blue: stationary solution without interaction; in green:  stationary solution for retarded interaction term without cut-off, in red: stationary solution for retarded interaction term with cut-off $H=0.1$.}
\label{figstat}\end{figure}  

\subsubsection{Convergence to equilibrium}

In this subsection we invstigate the convergence to equilibrium in more detail. In Figures~\ref{figfleece1} and \ref{figfleece2},  we investigate the behaviour (in time) of the $L_2$-norm of the distance to equilibrium for the spatial density $\rho$.  Again, the case without interaction is compared to the case with interaction with cut-off given by  $H\in\{0,0.1,0.5,\infty\}$. In Figure~\ref{figfleece1}, the case $A=1.0$ is considered, while Figure~\ref{figfleece2} shows the case $A=0.5$. 

\begin{figure}[h]
\includegraphics[clip,width=0.5\textwidth]{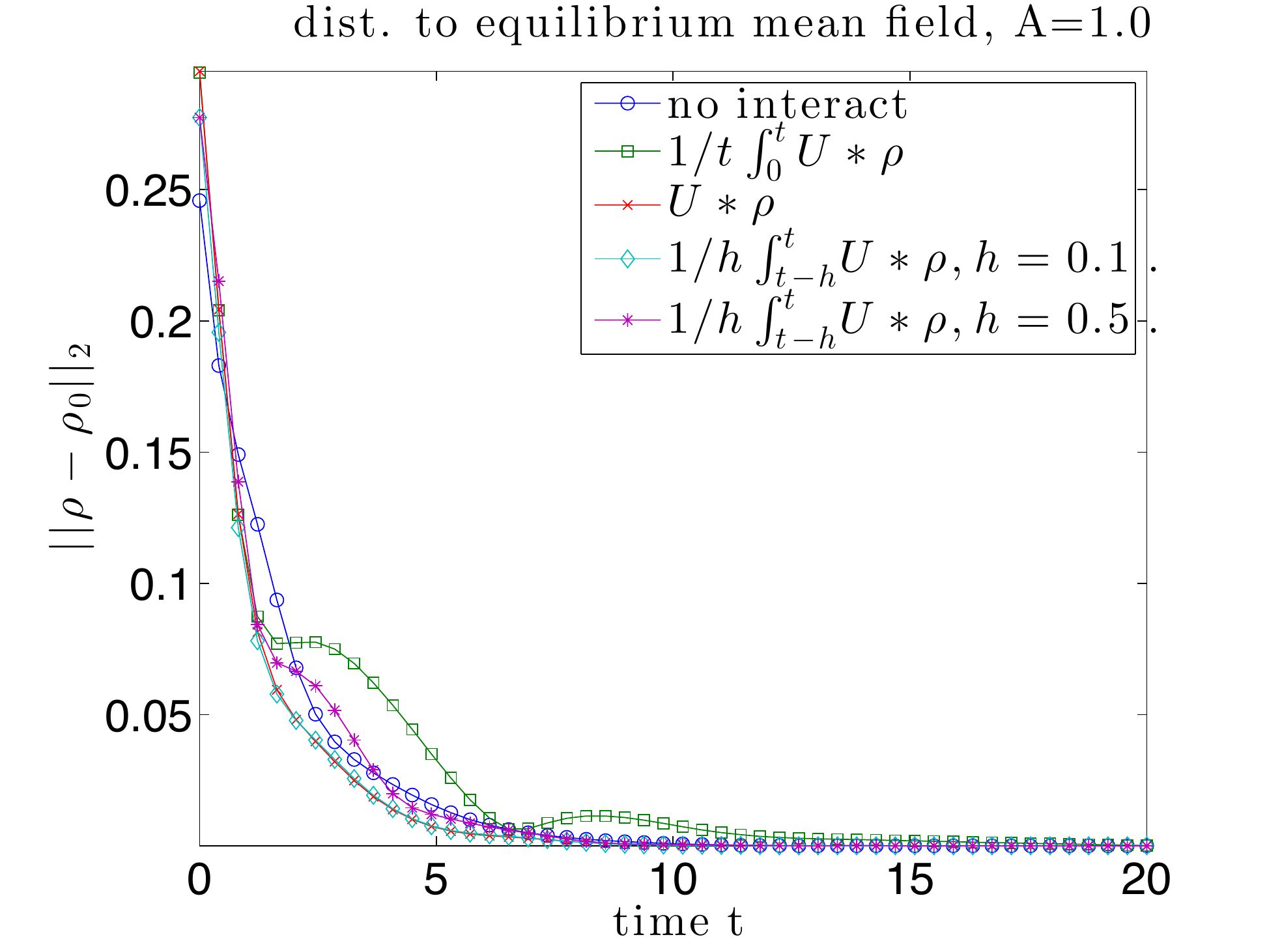}
\includegraphics[clip,width=0.5\textwidth]{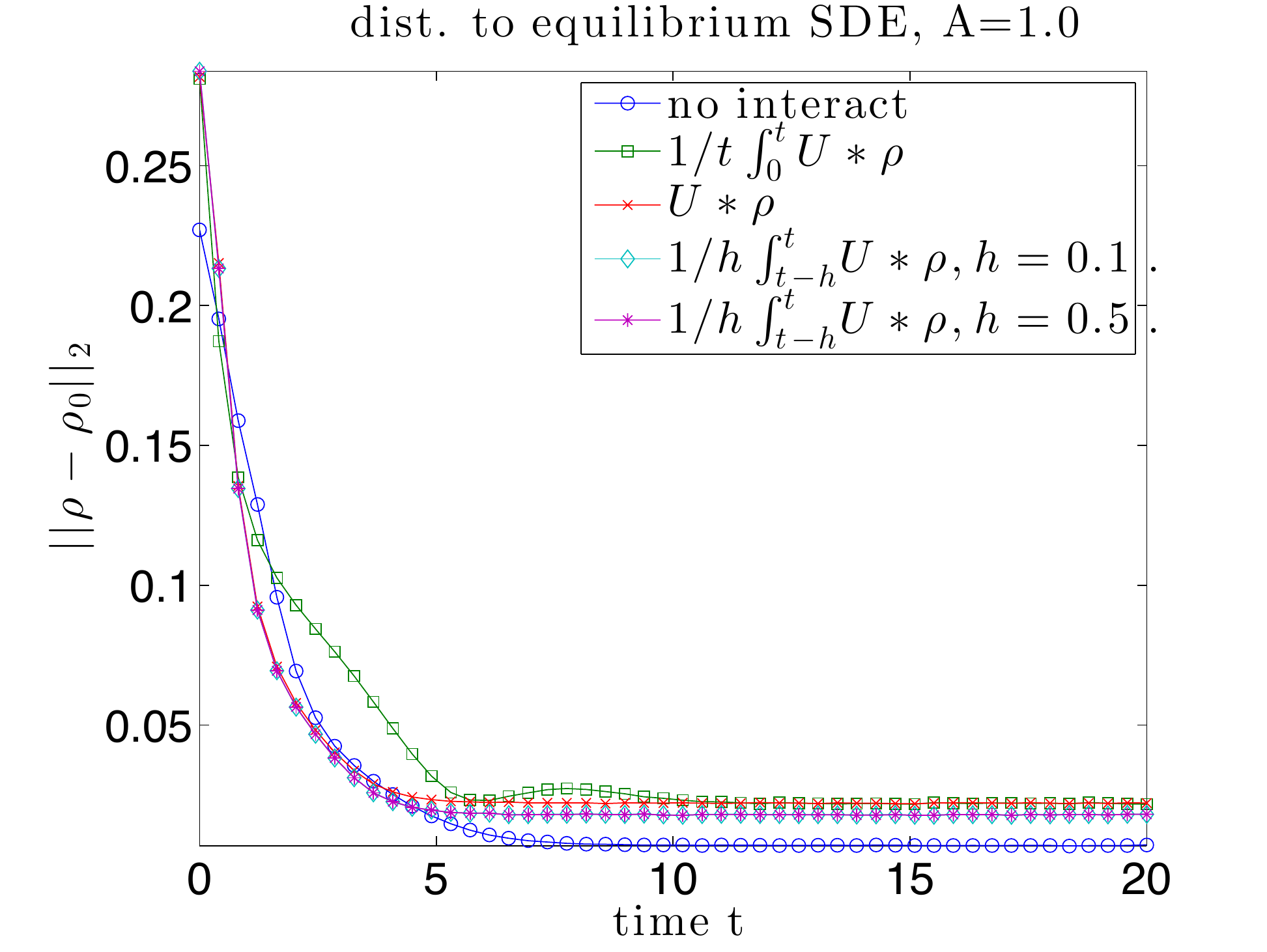}
\caption{In blue: decay without interaction; in green:  decay for retarded interaction term without cut-off, in red: decay for retarded interaction term with different $H$ and $A=1.0$.}
\label{figfleece1}
\end{figure}

The figures provide a strong indication that all cases with interaction converge to the same stationary distribution. We also observe the exponential decay to equilibrium for the case without interaction (cf.~\cite{DKMS12} for a theoretical exposition). Furthermore, the graph in green shows the case with a strongly retarded potential. The decay is no longer strictly monotone and it takes more time to achieve the stationary state, in comparison to the cases without significant delay.

\begin{figure}[h]
\includegraphics[clip,width=0.5\textwidth]{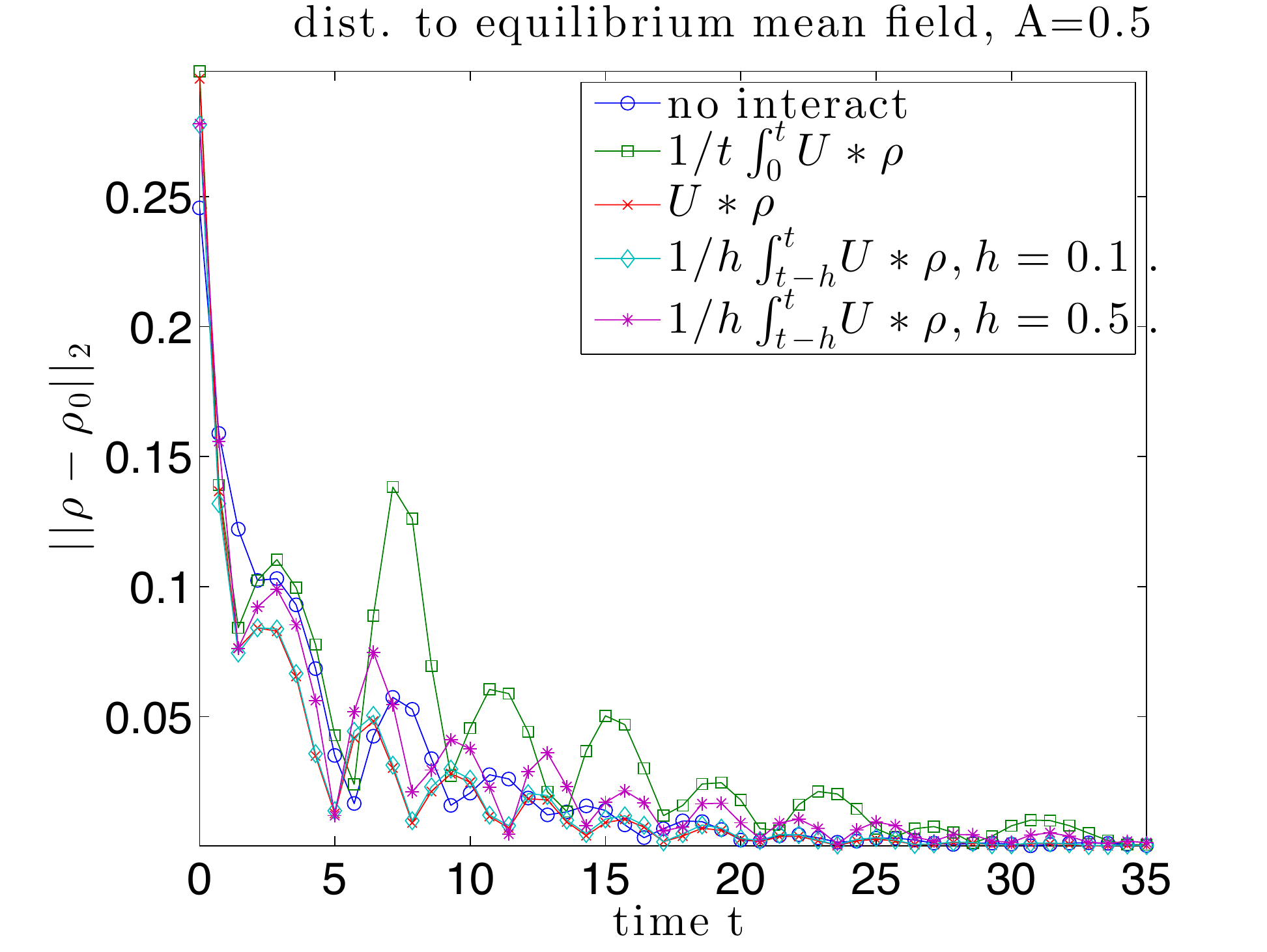}
\includegraphics[clip,width=0.5\textwidth]{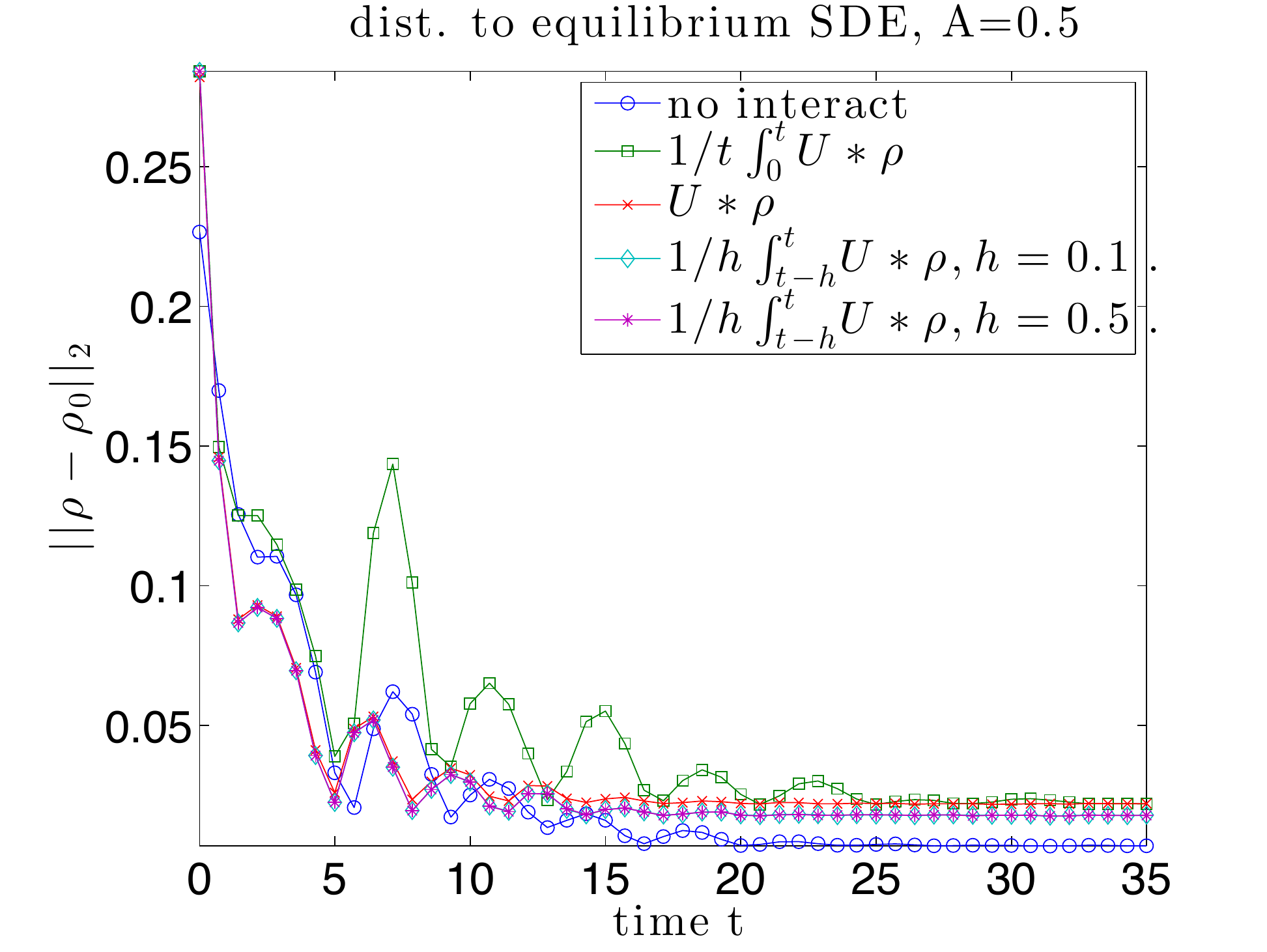}
\caption{In blue: decay without interaction; in green:  decay for retarded interaction term without cut-off, in red: decay for retarded interaction term with different $H$ and $A=0.5$.}
\label{figfleece2}
\end{figure}

Figure \ref{A} displays a comparison between the decay to equilibrium for different values of $A$. As can be observed, the decay becomes slower and, in particular, less and less regular for smaller values of $A$.

\begin{figure}[h]
\centering
\includegraphics[clip,width=0.6\textwidth]{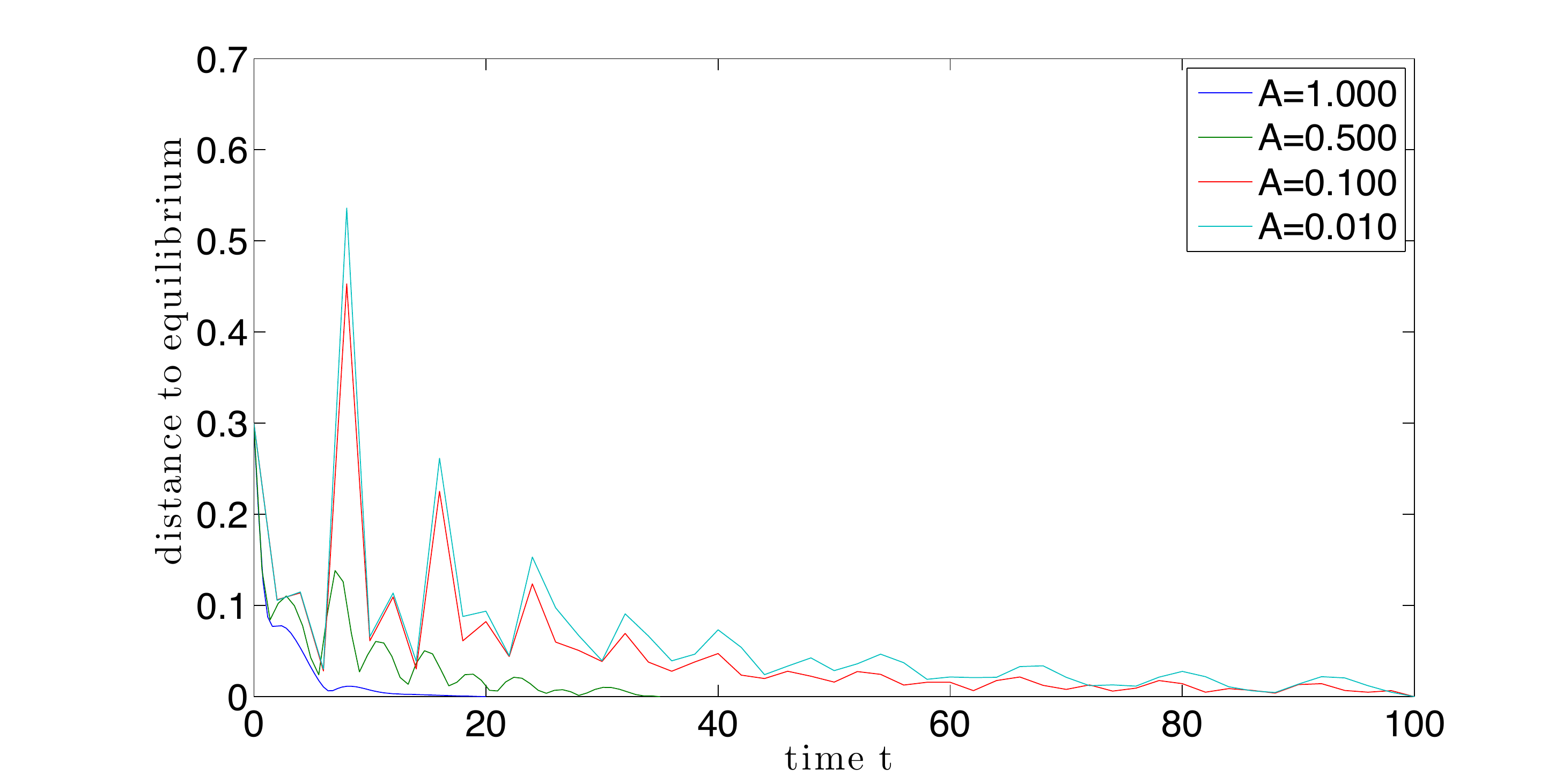}
\caption{Comparison of decay to equilibrium for different values of $A$ with $H=\infty$.}
\label{A}
\end{figure}


\subsubsection{Sensitivity of stationary solutions with respect to parameters in the interaction potential}

Finally, we investigate  the behavior of the stationary solution for different  parameters in the interaction model. In particular, we expect a wider profile for larger radii, as well as for stronger interaction, i.e., $C$ large. To numerically analyze the behavior for the different parameters, we consider the stationary equation in fixed-point form (\ref{eq:fixedpoint}). We are interested in the change of the stationary solution for varying parameters $R$ and $C$ respectively. 

In Figure~\ref{fig3fleece}, the  stationary solution is shown. First, we consider the case of varying radius $R$. We use the potential $U$ from (\ref{Uinteraction3}) and keep $C$ and $k$ fixed. As we can see in Figure~\ref{fig3fleece}, the solution $\rho$ is becoming less concentrated  for increasing radius $R$.
Second we vary  $C$ keeping $R$ and $k$ fixed obtaining similiar results as before.

\begin{figure}
\includegraphics[clip,width=0.5\columnwidth]{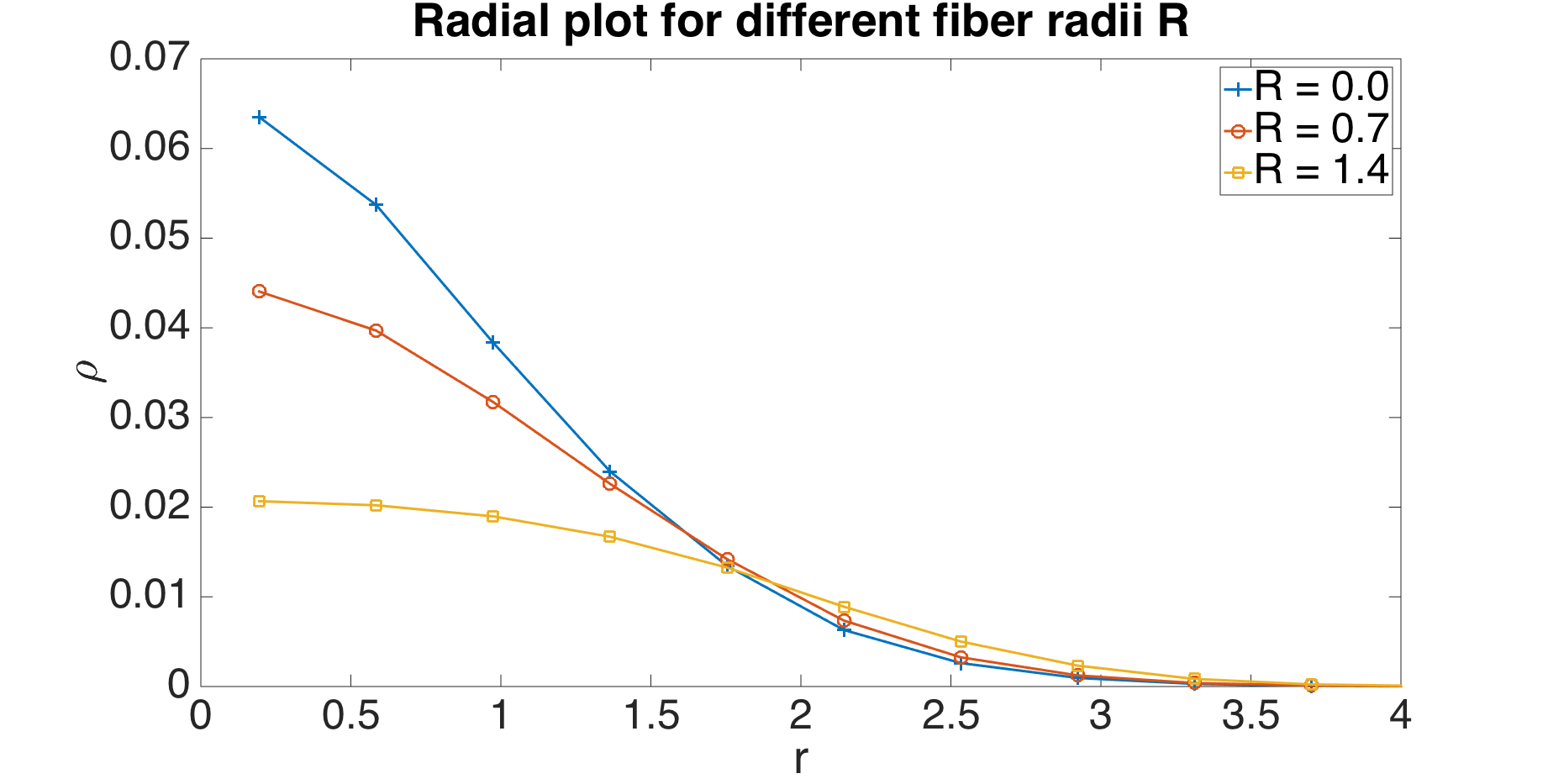}
\includegraphics[clip,width=0.5\columnwidth]{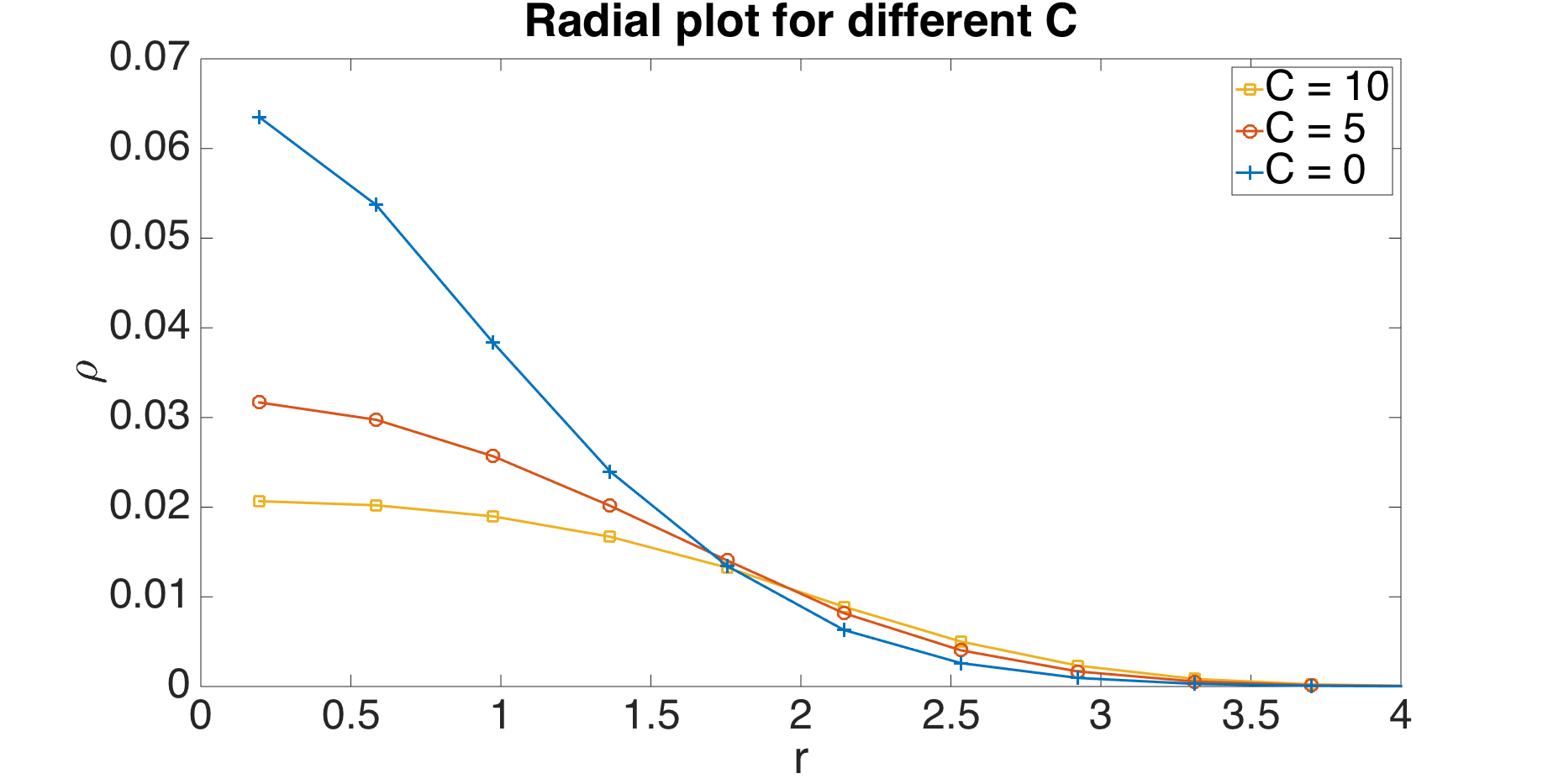}
\caption{Stationary solutions  $\rho$ for different values of the radius $R$  and different values of  the factor $C$ in the  interaction potential.}
\label{fig3fleece}
\end{figure}  

\section{Conclusion and outlook}\label{conclusion}
In this work we extended existing models for describing the lay-down process of fibers during the production of non-wovens to a model, which captures the interaction between the fibers by using a system of retarded stochastic differential equations and their corresponding mean-field approximations. A theoretical investigation of the mean-field limit for the deterministic retarded system is included. The solutions of the (retarded) microscopic and mean-field approximations are numerically compared with each other. In particular, the stationary solutions and  the decay to equilibrium is numerically investigated. Different types of interaction potentials are presented and the influence of different parameters in the interaction potential is investigated. Furthermore, a large diffusion scaling for the mean-field equations has been considered. 
Another interesting theoretical issue which has not been treated in the present work is the analytical investigation of the convergence of the retarded mean-field equation towards equilibrium. 

\appendix

\section{Proof of Proposition~\ref{prop:exists}}
Here we give a proof of Proposition \ref{prop:exists}, compare \cite{G12}. We begin by writing the differential equation in its integral form
 \begin{align}\label{eq:integral}
  Z_t(z) = z + \int_0^t a(Z_s(z))\,ds + \int_0^t \frac{1}{h(s)}\int_{s-h(s)}^s \K\mu_\sigma(Z_s(z))\,d\sigma\,ds.
 \end{align}
 Let $T>0$ be arbitrary but fixed and set $\B_T:=\C([0,T]\times\R^m;\R^m)$. We now denote $\mathcal{T}$ to be the operator on the right hand side, and show the existence of a unique fixed point for this operator. 
 
 To see that this operator is well-defined, let $Z\in \B_T$. The first to terms on the right hand side are obviously well-defined. Let us now consider the third term.
 By definition of the push forward measure, $\mu = Z_\#\mu_0 \in \C([0,T];\text{w--}\P_1(\R^m))$. Furthermore, since
 \[
  |\K\mu_t(z_1) - \K\mu_t(z_2)| \le \int_{\R^m} |B(z_1,\hat{z}) - B(z_2,\hat{z})|\,\mu_t(d\hat{z}) \le \Lip(B)|z_1-z_2|,
 \]
 for every $t\in[0,T]$, we infer that $\K\mu\in \C([0,T];\Lip_b(\R^m))\hookrightarrow \B_T$. A simple reformulation of the term under the integral leads to
 \begin{align*}
  \frac{1}{h(s)}\int_{s-h(s)}^s \K\mu_\sigma(Z_s(z))\,d\sigma &= \int_0^1 \K\mu_{\alpha(s,h(s),\sigma)}(Z_s(z))\,d\sigma \\
  &= \int_0^1 \int_{\R^m}B(Z_s(z),Z_{\alpha(s,h(s),\sigma)}(\hat z))\,\mu_0(d\hat z)\,d\sigma,
 \end{align*}
 with $\alpha(s,h(s),\sigma) = (1-\sigma)(s-h(s))+\sigma s$. From the assumptions, we conclude that this term is continuous in $s\in[0,T]$ and Lipschitz in $z\in\R^m$. Consequently, the third term is well-defined. Altogether, we have that $\mathcal{T}\colon \B_T\to \B_T$ as required.
 
 As a first estimate, we obtain
 \begin{align*}
  |(\T Z - \T\hat{Z})_t(z)| &\le \Lip(F)\int_0^t |Z_s(z)-\hat Z_s(z)|\,ds \\
   &+ \Lip(B)\int_0^t \frac{1}{h(s)}\int_{s-h(s)}^s \int_{\R^m} |Z_\sigma(\hat{z})-\hat{Z}_\sigma(\hat{z})|\, \mu_0(d\hat z)\,d\sigma\,ds,
 \end{align*}
 from the Lipschitz continuity of the mappings on the right hand side. Denoting 
\[
 \tilde\E(t) = \sup_{z\in\R^m} |(\T Z - \T\hat{Z})_t(z)|\quad\text{and}\quad\E(t):=\sup_{z\in\R^m}|(Z-\hat Z)_t(z)|,
\] 
we obtain for the cases $H\in(0,\infty)$, $t\le H$ and $h(t)=t$, the following estimate:
\begin{align*}
 \tilde\E(t) &\le \Lip(F)\int_0^t \E(s)\,ds + \Lip(B)\int_0^t \frac{1}{s}\int_{0}^s \E(\sigma) \,d\sigma\,ds \\
 &= \Lip(F)\int_0^t \E(s)\,ds + \Lip(B)\int_0^t  (\ln(t)-\ln(s))\, \E(s)\,ds \\
 &\le \Lip(F)\int_0^t  (1+\ln(t)-\ln(s))\, \E(s)\,ds,
\end{align*}
where we used integration by parts in the equality.
With this estimate at hand, we construct a solution by the Picard iteration procedure, and show that the generated sequence converges. We, therefore, define the sequence $(Z^{(k)})_{k\ge 0}\subset\B_T$ recursively by $Z^{(k+1)}=\T Z^{(k)}$ for $k\ge 0$, and denote
\[
 \E_k(t) = \sup_{z\in\R^m}|(Z^{(k+1)}-Z^{(k)})_t(z)|,\quad g(t,s)=1+\ln(t)-\ln(s).
\]
Applying the above estimate iteratively leads to
\begin{align*}
 \E_k(t) &\le \Lip(F)\int_0^t g(t,t_1)\,\E_{k-1}(t_1)\,dt_1 \\
 &\le \Lip(F)^2\int_0^t g(t,t_1)\int_0^{t_1} g(t_1,t_2)\,\E_{k-2}(t_2)\,dt_2\,\,dt_1 \\
 &\cdots\\
 &\le \Lip(F)^k\left[\int_0^t g(t,t_1)\.s \int_0^{t_{k-1}}g(t_{k-1},t_k) \,dt_k\.s dt_1\right] \sup_{0\le t_k\le t}\E_0(t_k).
\end{align*}
Elementary computations of the term in the bracket gives 
\[
 \E_k(t) \le \Lip(F)^k\frac{(k+1)t^k}{k!}\sup_{0\le t_k\le t}\E_0(t_k),
\]
for any $k\ge 0$. Since the series
\[
 \sum_{k\ge 0} c^k\frac{(k+1)t^k}{k!} = (1+ ct)e^{ct} <+\infty,
\]
summing up the terms $\E_k(t)$ in $k\ge 0$ yields
\[
 \sum_{k\ge 0} \E_k(t) \le \sup_{0\le t_k\le t}\E_0(t_k) (1+ \Lip(F)t)e^{\Lip(F)t} <+\infty,
\]
and hence $\E_k(t)\to 0$ pointwise in $t>0$. Consequently, the Picard sequence $(Z^{(k)})_{k\ge 0}$ converges uniformly in $\B_T$ to the solution of the integral equation (\ref{eq:integral}). Since $T$ was chosen arbitrarily, the solution may be extended to $Z\in\C(\R_+\times\R^m;\R^m)$.

As for uniqueness, we take two solutions $Z$ and $\hat Z$ of the integral equation, and denote its difference by
\[
 \E(t) = \sup_{z\in\R^m} |(Z-\hat Z)_t(z)|.
\]
From the estimates above, we have
\[
 \E(t) \le \Lip(F)\int_0^t  g(t,s)\, \E(s)\,ds.
\]
Let $T>0$ be arbitrary but fixed. Following the arguments from above, we obtain
\[
 \left(1 - \Lip(F)^k\frac{(k+1)t^k}{k!} \right)\sup_{0\le t\le T}\E(t) \le 0.
\]
Passing to the limit $k\to\infty$ yields $\E\equiv 0$ on $[0,T]$, and hence $Z=\hat Z$ on $[0,T]\times\R^m$. Since $T$ was chosen arbitrarily, this uniqueness result extends to $\R_+\times\R^m$.

We now consider the case $H\in(0,\infty)$. As above, we establish a unique solution $Y\in\B_{H}$. Setting $Y_{H}(z)=y$ as initial value, the integral form for $t\ge H$ reads
\[
 Z_t(y) = y + \int_{H}^t a(Z_s(y))\,ds +  \int_{H}^t \frac{1}{h(s)}\int_{s-h(s)}^s \K\mu_\sigma(Z_s(y))\,d\sigma\,ds.
\]
With the same notations as above, we obtain the following estimate
 \begin{align*}
  \tilde\E(t) &\le \Lip(F)\int_{H}^t \E(s)\,ds + \Lip(B)\int_{H}^t \frac{1}{h(s)}\int_{s-h(s)}^s \E(\sigma) \,d\sigma\,ds.
 \end{align*}
 The second term on the right hand side may be expressed in the form
 \begin{align*}
  \int_{H}^t \frac{1}{h(s)}\int_{s-h(s)}^s \E(\sigma) \,d\sigma\,ds &= \frac{t}{H}\int_{t-H}^t \E(s)\,ds - \int_0^{H}\E(s)\,ds \\
  &\hspace*{6em} -\ \frac{1}{H}\int_{H}^t s(\E(s)-\E(s-H))\,ds \\
  &=\frac{t}{H}\int_{t-H}^t \E(s)\,ds - \int_0^{H}\E(s)\,ds \\
  &\hspace*{4em}-\frac{1}{H}\int_{H}^ts\E(s)\,ds + \frac{1}{H}\int_0^{t-H}(s+H)\E(s)\,ds \\
  & \le \frac{t}{H}\int_0^t \E(s)\,ds,
 \end{align*}
 where we used integration by parts in the first equality, and the fact that
 \[
  \int_{H}^t (H-s)\E(s)\,ds \le 0,\quad \int_0^{t-H}(s-t)\E(s)\,ds \le 0,\quad \int_0^{t-H}\E(s)\,ds \le \int_0^t \E(s)\,ds,
 \]
 in the last inequality. Consequently, there is a constant $c=c(F,H)>0$, such that
 \[
  \tilde\E(t) \le c\,(1+t)\int_0^t\E(s)\,ds.
 \]
 Constructing a sequence $(Z^{(k)})_{k\ge 0}\in\B_T$ via the Picard iteration for some arbitrary but fixed $T>H$, with $Z^{(k)}_t\equiv Y_t$ for $t\in[0,H]$, we obtain as above
 \begin{align}\label{prop:exists:eq:2}
  \E_k(t)\le (1+t)c^k\left(\frac{t^k}{k!} + 2\frac{t^{k+1}}{(k+1)!}\right)\sup_{0\le t_k\le t}\E_0(t_k).
 \end{align}
 Notice that $\E_k\equiv 0$ on $[0,H]$. Replicating the arguments above yields existence and uniqueness of a solution $Z\in \B_T$. Moreover, we may extend the solution to $\R_+\times\R^m$.
 
 To conclude, we observe that 
 \begin{align*}
  \int_0^1 \int_{\R^m}B(Z^{(k)}_s(z),Z^{(k)}_{\alpha(s,h(s),\sigma)}(\hat z))\,\mu_0(d\hat z)\,d\sigma &\\
  &\hspace*{-6em}\longrightarrow\quad \int_0^1 \int_{\R^m}B(Z_s(z),Z_{\alpha(s,h(s),\sigma)}(\hat z))\,\mu_0(d\hat z)\,d\sigma,
 \end{align*}
 for $k\to \infty$ by the Lebesgue dominated convergence, and so 
 \[
  s\longmapsto \frac{1}{h(s)}\int_{s-h(s)}^s \K\mu_\sigma(Z_s(z))\,d\sigma
 \]
 is continuous for all $z\in\R^m$. Consequently, $t\mapsto Z_t(z)$, $t>0$, is continuously differentiable for all $z\in\R^m$, and satisfies the differential form (\ref{eq:meanfieldode}). \hfill\endproof


\begin{thebibliography}{10}

\bibitem{BS07}
Barrett, J.W., S\"uli,E.: {\em Existence of global weak solutions to some regularized
kinetic models of dilute polymers}.  SIAM Multiscale Mod. Simul. 6(2), pp. 506–546, 2007.

\bibitem{BS11}
Barrett, J.W., S\"uli, E.: {\em Finite element approximation of kinetic dilute polymer
models with microscopic cut-off}. ESAIM: Mathematical Modelling and Numerical Analysis, 45 (1), 39-89, 2011.


\bibitem{BLP78} Bensoussan, A., Lions, J.-L., Papanicolaou, G.: {\em Asymptotic analysis for periodic structures}, North Holland, 1978

\bibitem{BGKMW07}
Bonilla, L., G\"otz, T., Klar, A., Marheineke, N., Wegener, R.:
\emph{Hydrodynamic limit for the Fokker-Planck equation of fiber lay-down models},
SIAM Appl. Math. 68, 648--655, 2007.

\bibitem{BH} 
Braun, W., Hepp, K.: 
\emph{The Vlasov dynamics and its fluctuations in the 1/N limit of interacting classical particles}, 
Commun. Math. Phys. 56, 101--113, 1977.

\bibitem{BCC}
Bolley, F., Canizo, J.A., Carrillo, J.A.: 
\emph{Stochastic Mean-Field Limit: Non-Lipschitz Forces and Swarming},
Math. Mod. Meth. Appl. Sci. 21, 2179-2210, 2011.


\bibitem{CKMT}
Carrillo, J. A., Klar, A., Martin, S., Tiwari, S.: 
\emph{Self-propelled interacting particle systems with roosting force},
Math. Models Methods Appl. Sci. 20, 1533--1552, ISSN:0218-2025, 2010.

\bibitem{CDP}
Carrillo, J.A., D'Orsogna, M.R., Panferov, V.: 
\emph{Double milling in self-propelled swarms from kinetic theory},
KRM 2, 363--378, 2009.


\bibitem{J.A.Carrillo2014}
Carrillo, J.A., Klar, A., Roth, A.:
\emph{Single to double mill small noise transition via semi-Lagrangian finite volume methods},
http://arxiv.org/abs/1407.5051

\bibitem{C.Cheng1976}
{Cheng, C., Knorr, G.:}
\emph{The integration of the Vlasov equation in configuration space},
Journal of Computational Physics 22, 330--351, 1976.

\bibitem{DM1} 
Degond, P., Motsch, S.:
\emph{Continuum limit of self-driven particles with orientation interaction}, 
Math. Models Methods Appl. Sci. 18, 1193--1215, 2008.

\bibitem{dobru}
Dobrushin, R.: 
\emph{Vlasov equations}, 
Funct. Anal. Appl. 13, 115--123, 1979.

\bibitem{DKMS12}
{Dolbeault, J., Klar, A., Mouhot, C., Schmeiser, C.} 
\emph{Hypocoercivity and a Fokker-Planck equation for fiber lay-down},
Applied Mathematical Research Express 2012, doi10.1093/amrx/abs015.

\bibitem{Sem2}
{Douglas, J., Huang, C., Pereira, F.:}
\emph{The modified method of characteristics with adjusted advection},
Numerische Mathematik 83, 353--369, 1999.

\bibitem{G03}
Golse, F.: 
\textit{The mean field limit for the dynamics of large particle systems}.
Journ\'ees \'equations aux d\'eriv\'ees partielles, {\bf 9} (2003), pp. 1--47.


\bibitem{G12}
Golse, F.:
\emph{The mean field limit for a regularized Vlasov-Maxwell dynamics},
Commun. Math. Phys. 310, 789-–816, 2012.

\bibitem{GMR14}
Golse, F., Mouhot, C., Ricci, V.: 
\emph{Empirical measures and Vlasov hierarchies},
http://arxiv.org/abs/1309.0222.

\bibitem{GKMW07}
G\"otz, T., Klar, A., Marheineke, N., Wegener, R.:
\emph{A stochastic model and associated Fokker--Planck equation for the fiber lay-down process in nonwoven production processes},
SIAM Appl. Math 67, 1704--1717, 2007.

\bibitem{HKMO09} 
Herty, M., Klar, A., Motsch, S., Olawsky, F.:
\emph{A smooth model for fiber lay-down processes and its diffusion approximations},
KRM 2, 480--502, 2009.

\bibitem{KantoRubin1958}
{Kantorovich, L.V., Rubinshtein, G.S.:}
{\em On the space of totally additive functions},
Vestn. Leningrad. Univ. 13(7), 52--59, 1958.

\bibitem{KMW12}
{Klar, A., Maringer, J., Wegener, R.:}
{\em A 3D model for fiber lay-down processes in non-woven production processes},
Math. Models Methods Appl. Sci. 2012.
  
\bibitem{KMW12b}
{Klar, A., Maringer, J., Wegener, R.:}
{\em A smooth 3D model for fiber lay-down processes in non-woven production processes},
KRM 5(1), 97--112, 2012.
 

\bibitem{KRS07}
{Klar, A., Reutersw\"ard, P., Sea\"id, M.:}
\emph{A semi-Lagrangian method for a Fokker-Planck equation describing fiber dynamics},
Journal of Scientific Computing 38, 349--367, 2009.

\bibitem{KST14}
{Klar, A., Schneider, F., Tse, O.:}
\emph{Approximate models for stochastic dynamic systems with velocities on the sphere and associated Fokker--Planck equations},
KRM 7(3), 509--529, 2014.

\bibitem{neunzert}
Neunzert, H.: 
\emph{The Vlasov equation as a limit of Hamiltonian classical mechanical systems of interacting particles}. 
Trans. Fluid Dynamics 18, 663--678, 1977.

\bibitem{Sem3}
{Qiu, J.-M., Christlieb, A.:}
\emph{A Conservative high order Semi-Lagrangian WENO method for the Vlasov Equation},
Journal of Computational Physics 229, 1130--1149, 2010.

\bibitem{RKSZ14}
Roth, A., Klar, A., Simeon, B., Zharovsky, E.: 
\emph{ A semi-Lagrangian finite volume method for a 3-D Fokker-Planck equations associated to stochastic dynamical systems on the sphere}, 
J. Scientific Comp. 61(3), DOI 10.1007/s10915-014-9835-z, 2014.


\bibitem{R14}
Roth, A.: \emph{Numerical Schemes for Kinetic Equations with Applications to Fibre Lay-Down and Interacting Particles}, Verlag Dr. Hut, München, ISBN 978-3-8439-1936-4, 2015.


\bibitem{AltGeo3}
{Sadourny, R., Arakawa, A., Mintz, Y.}:
\emph{Integration of the nondivergent barotropic vorticity equation with an icosahedral-hexagonal grid for the sphere},
Monthly Weather Review 96, 351--356, 1968.

\bibitem{Sem6}
{Sonnendr{\"u}cker, E., Roche, J., Bertrand, P., Ghizzo, A.}:
\emph{The Semi-Lagrangian Method for the Numerical Resolution of the Vlasov Equation},
Journal of Computational Physics 149, 201--220, 1999.


\bibitem{spohn2}
Spohn, H.: 
\emph{Large scale dynamics of interacting particles},
Texts and Monographs in Physics, Springer, 1991.

\bibitem{Villani2008}
{Villani, C.:}
{Optimal transport, old and new}, 
Grundlehren der Mathematischen Wissenschaften 338, Springer-Verlag, 2008.
\end{thebibliography}
\end{document}